\documentclass[preprint,11pt]{elsarticle}
\usepackage{amsmath,amssymb}
\usepackage[all]{xy}
\usepackage{latexsym}
\usepackage{amsthm,color}
\usepackage{amsmath,amscd,verbatim}
\usepackage{hyperref}
\usepackage{graphicx}
\usepackage{tkz-euclide,pgfplots}

\usepackage[total={150mm,225mm}]{geometry}

\theoremstyle{plain}
  \newtheorem{thm}{Theorem}[section]
  \newtheorem{lem}[thm]{Lemma}
  \newtheorem{prop}[thm]{Proposition}
  \newtheorem{cor}[thm]{Corollary}
\theoremstyle{definition}
  \newtheorem{defn}[thm]{Definition}
  
  \newtheorem{exmp}[thm]{Example}
  \newtheorem{rem}[thm]{Remark}

\newtheorem*{con}{Convention}
\newtheorem*{SA}{Standing Assumption}

\makeatletter \def\ps@pprintTitle{  \let\@oddhead\@empty  \let\@evenhead\@empty  \def\@oddfoot{\centerline{\thepage}} \let\@evenfoot\@oddfoot} \makeatother

\begin{document}

\newcommand{\oto}{{\lra\hspace*{-3.1ex}{\circ}\hspace*{1.9ex}}}

\newcommand{\lam}{\lambda}
\newcommand{\da}{\downarrow}
\newcommand{\Da}{\Downarrow\!}
\newcommand{\ua}{\uparrow}
\newcommand{\ra}{\rightarrow}
\newcommand{\la}{\leftarrow}
\newcommand{\lra}{\longrightarrow}
\newcommand{\lla}{\longleftarrow}
\newcommand{\rat}{\!\rightarrowtail\!}
\newcommand{\up}{\upsilon}
\newcommand{\Up}{\Upsilon}
\newcommand{\Lam}{\Lambda}
\newcommand{\CF}{{\cal F}}
\newcommand{\CG}{{\cal G}}
\newcommand{\CH}{{\cal H}}
\newcommand{\CN}{{\mathcal{N}}}
\newcommand{\CB}{{\cal B}}
\newcommand{\CI}{{\cal I}}
\newcommand{\CT}{{\cal T}}
\newcommand{\CS}{{\cal S}}
\newcommand{\CV}{\mathfrak{V}}
\newcommand{\CU}{\mathfrak{U}}
\newcommand{\CP}{{\cal P}}
\newcommand{\CQ}{\mathcal{Q}}
\newcommand{\nb}{{\rm int}}
\newcommand{\bv}{\bigvee}
\newcommand{\bw}{\bigwedge}
\newcommand{\dda}{\downdownarrows}
\newcommand{\dia}{\diamondsuit}
\newcommand{\y}{{\bf y}}
\newcommand{\colim}{{\rm colim}}
\newcommand{\fR}{R^{\!\forall}}
\newcommand{\eR}{R_{\!\exists}}
\newcommand{\dR}{R^{\!\da}}
\newcommand{\uR}{R_{\!\ua}}
\newcommand{\swa}{{\swarrow}}
\newcommand{\sea}{{\searrow}}
\newcommand{\id}{{\rm id}}
\newcommand{\cl}{{\rm cl}}
\newcommand{\sub}{{\rm sub}}

\numberwithin{equation}{section}
\renewcommand{\theequation}{\thesection.\arabic{equation}}

\begin{frontmatter}
\title{A comparative study of ideals in fuzzy orders\tnoteref{F}} \tnotetext[F]{This work is supported by National Natural Science Foundation of China (No.  11771310)}

\author{Hongliang Lai}
\ead{hllai@scu.edu.cn}

\author{Dexue Zhang\corref{cor}}\cortext[cor]{Corresponding author.}
\ead{dxzhang@scu.edu.cn}

\author{Gao Zhang}
\ead{gaozhang0810@hotmail.com}
\address{School of Mathematics, Sichuan University, Chengdu 610064, China}

\begin{abstract}
This paper presents a comparative study of three kinds  of ideals in fuzzy order theory: forward Cauchy ideals (generated by forward Cauchy nets), flat ideals  and irreducible ideals, including their role in  connecting  fuzzy order with fuzzy topology.
\end{abstract}


\end{frontmatter}

\section{Introduction}
The notion of ideals  (i.e., directed lower sets) in   ordered sets  is  primitive  in domain theory. Domains  and Scott topology are both postulated in terms of ideals and their suprema. For a partially ordered set $P$, let ${\rm Idl}(P)$ denote the set of ideals in $P$ with the inclusion order, and $\y:P\lra {\rm Idl}(P)$ be the map that assigns each $x\in P$ to the principal ideal $\da\!x$. Then $P$ is directed complete if $\y$ has a left adjoint $\sup:{\rm Idl}(P)\lra P$  (which sends each ideal to its supremum); $P$ is a domain if it is directed complete and the left adjoint of $\y$  also has a left adjoint.  A Scott open set of $P$ is an upper set $U$ such that for each ideal $I$ in $P$, if the supremum of $I$ is in $U$ then $I$ intersects with $U$.

In order to establish a theory of fuzzy domains (or, quantitative domains), the first step is to find an appropriate notion of ideals for fuzzy orders (or, $\CQ$-orders, where $\CQ$ is the truth-value quantale). The problem seems simple, but, it turns out to be a very intricate one because of  the complication of the table of truth-values --- the quantale $\CQ$. In fact,  there are several natural extension of this notion to the fuzzy setting.  This paper presents a comparative study of three kinds of them: forward Cauchy ideals, flat ideals and irreducible ideals.

Before summarizing related attempts in the literature and explaining what we will do in this paper, we recall some equivalent reformulations of ideals in a partially ordered set. Let $P$ be a partially ordered set. A net $\{x_i\}$ in  $P$ is  eventually monotone  if there is some $i$ such that $x_j\leq x_k$ whenever $i\leq j\leq k$ \cite{Gierz2003}.
Let $I$ be a non-empty lower set in $P$. The following are equivalent: \begin{itemize} \setlength{\itemsep}{-2pt} \item $I$ is an ideal, that is, for any $x,y$ in $I$, there is some $z\in I$ such that $x,y\leq z$. \item There exists an eventually monotone net $\{x_i\}$ such that $I= \bigcup_i\bigcap_{j\geq i}\downarrow\!x_j$. \item $I$ is  \emph{flat}  in the sense that for any upper sets $G,H$ of $P$, if $I$ intersects with both $G$ and $H$, then $I$ intersects with $G\cap H$. \item $I$ is  \emph{irreducible}  in the sense that for any lower sets $B,C$ of $P$, if $I\subseteq B\cup C$ then either $I\subseteq B$ or $I\subseteq C$. \end{itemize}

The net-approach is extended to fuzzy orders in \cite{BvBR1998,Wagner94,Wagner97}, resulting in the notions of forward Cauchy net and Yoneda completeness (a.k.a liminf completeness). Fuzzy lower sets generated by forward Cauchy nets are called ideals in \cite{FK97,FS02}. They will be called forward Cauchy ideals in this paper, in order to distinguish them from flat ideals and irreducible ideals. Yoneda completeness, as a version of quantitative directed completeness, has received wide attention in the study of fuzzy orders, including generalized metric spaces as a special case, see e.g. \cite{BvBR1998,FK97,FS02,FSW,Goubault,HW2011,HW2012, KS2002,LZ16,Ru,Wagner97}.

The extension of the flat-approach to the fuzzy setting originates in the work of Vickers \cite{SV2005} in the case the truth-value quantale is Lawvere's quantale $([0,\infty]^{\rm op},+)$ (which is isomorphic to the unit interval with the product t-norm). This approach results in the notions of flat ideal  (called flat left module  in \cite{SV2005}) and flat completeness of fuzzy orders. It is shown in \cite{SV2005} that for Lawvere's quantale,  flat completeness is equivalent to Yoneda completeness.

The recent paper \cite{Zhang18} extends the irreducible-approach to the fuzzy setting in the study of sobriety of fuzzy cotopological spaces, resulting in the notions of irreducible ideal  and irreducible completeness of fuzzy orders. 

Forward Cauchy ideals, flat ideals and irreducible ideals in a fuzzy ordered set are all natural generalizations of the notion of ideals in a partially ordered set; the resulting completeness for fuzzy orders are  natural extensions of directed completeness in order theory.  We note in passing that, from a category theory perspective, such completeness for fuzzy orders is an example of the theory of  cocompleteness in enriched category theory with respect to a class of weights \cite{AK,Kelly,KS05}.

This paper aims to present a comparative study of forward Cauchy ideals, flat ideals and irreducible ideals, hence of the resulting completeness notions. Since all of them are intended to play the role of directed lower sets in fuzzy order theory, before comparing them with each other, we propose the following  criteria for a class $\Phi$ of fuzzy sets  that are meant for the role of ideals in fuzzy orders:
\begin{enumerate}
\item[(I1)] If the truth-value quantale $\CQ$ is the two-element Boolean algebra, then for each partially ordered set $A$, $\Phi(A)$ is the set of ideals in $A$. This is to require that $\Phi$ is a generalization of the class of ideals.

\item[(I2)] $\Phi$ is saturated. Saturatedness of $\Phi$ guarantees that for each $\CQ$-ordered set $A$, $\Phi(A)$ is the free $\Phi$-continuous $\CQ$-ordered set generated by $A$. So, for a saturated class $\Phi$ of fuzzy sets, there exist enough $\Phi$-continuous $\CQ$-ordered sets.

\item[(I3)] $\Phi$ generates a  functor from the category of $\CQ$-ordered sets and $\Phi$-cocontinuous maps to that of $\CQ$-topological spaces and/or    $\CQ$-cotopological spaces.  This functor is expected to play the role of the functor in domain theory that sends each partially ordered set to its Scott topology. 
    As in the classical case, such  functors are of fundamental importance in the theory of fuzzy domains. \end{enumerate}

Besides the interrelationship among the classes of forward Cauchy ideals, flat ideals and   irreducible ideals, their connection to fuzzy topology will also be discussed.

The contents are arranged as follows.

Section 2 recalls some basic ideas that are needed in the subsequent sections.

Section 3 concerns the relationship among forward Cauchy ideals, flat ideals and  irreducible ideals. The main results are: (i) Every forward Cauchy ideal is flat (irreducible, resp.) if and only if the truth-value quantale is meet continuous (dually meet continuous, resp.).   (ii) For the quantale obtained by equipping $[0,1]$   with a left continuous t-norm,   irreducible ideals coincide with  forward Cauchy ideals. (iii) For a prelinear quantale, every irreducible ideal  is  flat. (iv) For a quantale that satisfies the law of double negation, flat ideals coincide with irreducible ideals.

Section 4 proves that for every quantale, both the class of flat ideals and that of irreducible ideals are saturated. As for forward Cauchy ideals, it is shown in \cite{FSW} that for a completely distributive value quantale (see \cite{FS02,FSW}  for definition), the class of forward Cauchy ideals is saturated. The conclusion is extended in \cite{LZ07} to the case that $\CQ$ is a continuous and integral quantale.

Section 5 concerns the connection between fuzzy orders and fuzzy topological spaces. For each subclass $\Phi$ of flat ideals, a functor is constructed from the category of $\CQ$-ordered sets and $\Phi$-cocontinuous maps to that of stratified $\CQ$-topological spaces.  For each subclass $\Phi$ of irreducible ideals, a full  functor is constructed from the category of $\CQ$-ordered sets and $\Phi$-cocontinuous maps to that of stratified $\CQ$-cotopological spaces. This shows that irreducible ideals can be used to generate closed sets, hence $\CQ$-cotopologies, whereas flat ideals can be used to generate open sets, hence $\CQ$-topologies. We would like to remind the reader that, in general, there is no natural way to switch between closed sets and open sets in the fuzzy setting. This lack of ``duality"  between closed sets and open sets demands that we need different kinds of fuzzy ideals to connect fuzzy orders with fuzzy topological spaces and/or fuzzy cotopological spaces. This is the \emph{raison d'etre} for flat ideals and irreducible ideals.

\section{Preliminaries}
In this preliminary section, we recall briefly some basic ideas of complete lattices \cite{Gierz2003}, quantales \cite{Rosenthal1990}, and $\CQ$-orders that will be needed.

A quantale $\mathcal{Q}$ is a monoid in the monoidal category of complete lattices and join-preserving maps \cite{Rosenthal1990}. Explicitly, a quantale $\mathcal{Q}$ is a monoid $(Q,\&)$ such that $Q$ is a complete lattice and
\begin{equation*} p\&\bv_{j\in J}q_j=\bv_{j\in J}p\& q_j, ~ \Big(\bv_{j\in J}q_j\Big)\&p=\bv_{j\in J}q_j\&p. \end{equation*}
for all $p\in Q$ and   $\{q_j\}_{j\in J}\subseteq Q$. The unit $1$ of the monoid $(Q,\&)$ is in general not   the top element of $Q$. If it happens that the unit element coincides with the top element of $Q$, then we say that $\mathcal{Q}$ is \emph{integral}. If the operation $\&$ is commutative then we say $\mathcal{Q}$ is a commutative quantale. A quantale $(Q,\&)$ is  meet continuous if the underlying lattice $Q$ is    meet  continuous.

\begin{SA} Throughout this paper, if not otherwise specified, all quantales are assumed to be integral and commutative.\end{SA}

Since the semigroup operation $\&$ distributes over arbitrary joins, it  determines a binary operation $\ra$ on $Q$ via the adjoint property \begin{equation*} p\&q\leq r\iff q\leq p\ra r. \end{equation*}
The  binary operation $\ra$ is called the \emph{implication}, or the \emph{residuation}, corresponding to  $\&$.

Some basic properties of the binary operations $\&$ and $\ra$ are collected below, they can be found in many places, e.g.  \cite{Belo02,Rosenthal1990}.

\begin{prop}\label{2.1} Let $\mathcal{Q}$ be a quantale. Then
\begin{enumerate}[(1)]
\item  $1\ra p=p$.
\item  $p\leq q \iff 1= p\ra q$.
\item  $p\ra(q\ra r)=(p\& q)\ra r$.
\item  $p\&(p\ra q)\leq q$.
\item  $\Big(\bv_{j\in J}p_j\Big)\ra q=\bw_{j\in J}(p_j\ra q)$.
\item  $p\ra\Big(\bw_{j\in J} q_j\Big)=\bw_{j\in J}(p\ra q_j)$. \item $p=\bw_{q\in Q}((p\ra q)\ra q)$.
\end{enumerate}
\end{prop}
We often write $\neg p$ for $p\ra 0$ and call it the \emph{negation} of $p$. Though it is  true that $p\leq\neg\neg p$ for all $p\in Q$, the inequality $\neg\neg p\leq p$ does not always hold.
A quantale $\mathcal{Q}$   satisfies the {\it law of double negation} if  \[(p\ra 0)\ra 0 = p \]    for all $p\in Q$.

\begin{prop}{\rm(\cite{Belo02})} \label{properties of negation} Suppose that $\CQ$ is a quantale that satisfies the law of double negation. Then \begin{enumerate}[(1)] \item $p\ra q = \neg(p\&\neg q)=\neg q\ra \neg p$. \item $p\&q =\neg (q\ra\neg p)=\neg (p\ra\neg q)$. \item $\neg(\bw_{i\in I}p_i) = \bv_{i\in I}\neg p_i$.  \end{enumerate}\end{prop}

The quantales with the unit interval $[0,1]$ as underlying lattice are of particular interest in fuzzy set theory \cite{Ha98,KMP00}. In this case, the semigroup operation $\&$ is exactly a left continuous t-norm on $[0,1]$ \cite{KMP00}. A continuous t-norm on $[0,1]$ is a left continuous t-norm $\&$ that is continuous  with respect to the usual topology.

\begin{exmp} (\cite{KMP00})
 Some basic   t-norms:
\begin{enumerate}[(1)]
\item The  t-norm $\min$: $a\&b=a\wedge b=\min\{a,b\}$. The corresponding implication is given by \[a\ra
b=\left\{\begin{array}{ll} 1, & a\leq b;\\
b, & a>b.\end{array}\right.\]
\item The product t-norm: $a\&b=a\cdot b$. The
corresponding implication  is given by $$a\ra
b=\left\{\begin{array}{ll} 1, & a\leq b;\\
b/a, & a>b.\end{array}\right.$$
\item The {\L}ukasiewicz t-norm:
$a\&b=\max\{a+b-1,0\}$. The corresponding implication is given by
$$a\ra b=
\min\{1, 1-a+b\}. $$ In this case, $([0,1],\&)$ satisfies the law of double negation. \item The nilpotent minimum t-norm: $$a\& b=\left\{\begin{array}{ll} 0, & a+b\leq 1;\\ \min\{a,b\}, & a+b>1.\end{array}\right.$$ The corresponding implication  is given by $$a\ra b=\left\{\begin{array}{ll} 1, & a\leq b;\\ \max\{1-a,b\}, & a>b.\end{array}\right.$$  In this case, $([0,1],\&)$ satisfies the law of double negation. \end{enumerate}\end{exmp}
The following theorem, known as the ordinal sum decomposition theorem, is of fundamental importance in the theory of continuous t-norms.

\begin{thm} {\rm(\cite{Fau55,Mostert1957})} \label{ordinal sum}
For each continuous t-norm $\&$  on $[0,1]$, there is a set of disjoint open intervals $\{(a_i,b_i)\}$ of $[0,1]$ that satisfy the following conditions: \begin{enumerate}[(i)] \item For each $i$, both $a_i$ and $b_i$ are idempotent and the restriction of $\&$ on $[a_i,b_i]$ is either isomorphic to the \L ukasiewicz t-norm or to the product t-norm; \item $x\&y=\min\{x,y\}$ if $(x,y)\notin\bigcup_i[a_i,b_i]^2$.  \end{enumerate}
\end{thm}

A \emph{$\mathcal{Q}$-order} (or an order valued in the quantale $\CQ$) \cite{Wagner94,Zadeh71} on a set $A$ is a reflexive and transitive $\mathcal{Q}$-relation on $A$. Explicitly, a $\mathcal{Q}$-order on  $A$  is a map $R: A\times A\lra Q$ such that $R(x,x)=1$   and  $R(y,z)\& R(x,y)\leq R(x,z)$ for any $x,y,z\in A$.  The pair $(A,R)$ is called a  $\mathcal{Q}$-ordered  set.  A $\CQ$-ordered set is also called a \emph{$\CQ$-category} in the literature,  since it is precisely a category enriched over the symmetric monoidal category $\CQ$. As usual, we write $A$ for the pair $(A, R)$ and $A(x,y)$ for $R(x,y)$ if no confusion would arise.

Two elements $x,y$ in a $\CQ$-ordered set $A$ are \emph{isomorphic} if $A(x,y)=A(y,x)=1$. We say that $A$ is \emph{separated} if  isomorphic elements in $A$ are equal, that is, $A(x,y)=A(y,x)=1$ implies that $x=y$.

If $R: A\times A\lra Q$ is a $\CQ$-order on $A$, then $R^{\rm op}: A\times A\lra Q$, given by $R^{\rm op}(x,y)=R(y,x)$, is also a $\CQ$-order on $A$ (by commutativity of $\&$), called the opposite of $R$.

\begin{exmp} This example belongs to the folklore in  fuzzy order theory, see e.g. \cite{Belo02}. For all $p, q\in Q$, let \[d_L(p,q)=p\ra q.\] Then $(Q,d_L)$ is a separated $\CQ$-ordered set. The opposite of $(Q,d_L)$ is   $(Q,d_R)$, where \[d_R(p,q)=q\ra p.\] Both $(Q,d_L)$ and $(Q,d_R)$ play important roles in the theory of $\CQ$-ordered sets. \end{exmp}

\begin{exmp} \cite{Belo02} Let $X$ be a set. A map $\lam: X\lra Q$ is called a  fuzzy set (valued in $\CQ$) of $X$, the value $\lam(x)$ is  interpreted as the membership degree of $x$.
The map \[\sub_X: Q^X\times Q^X\lra Q,\] given by
\begin{equation*}\sub_X(\lam, \mu)=\bw_{x\in X}\lam(x)\ra \mu(x),\end{equation*} defines a separated $\mathcal{Q}$-order on $Q^X$. Intuitively, the value $\sub_X(\lam,\mu)$ measures the degree that $\lam$ is a subset of $\mu$. Thus, $\sub_X$ is called the \emph{fuzzy inclusion order} on $Q^X$. The opposite of $\sub_X$ is called the \emph{converse fuzzy inclusion order} on $Q^X$. In particular, if $X$ is a singleton set then the $\mathcal{Q}$-ordered sets $(Q^X,\sub_X)$  and $(Q^X,\sub_X^{\rm op})$ degenerate to the $\mathcal{Q}$-ordered sets $(Q,d_L)$ and $(Q,d_R)$, respectively.  \end{exmp}

A map $f: A\lra B$ between $\mathcal{Q}$-ordered sets is $\mathcal{Q}$-order preserving if \[A(x_1,x_2)\leq B(f(x_1),f(x_2))\]  for any $x_1,x_2\in A$. We write \[\mbox{$\CQ$-{\sf Ord}}\] for the category of $\CQ$-ordered sets and $\mathcal{Q}$-order preserving maps.

Let $f: A\lra B$ and $g:B\lra A$ be $\CQ$-order preserving maps. We say $f$ is   left adjoint to $g$ (or, $g$ is   right adjoint to $f$), $f\dashv g$ in symbols, if $$A(x,g(y))=B(f(x),y)$$ for all $x\in A$ and $y\in B$.

Let $A,B$ be $\CQ$-ordered sets. A $\CQ$-distributor $\phi:A\oto B$ from  $A$ to $B$ is a map $\phi:A\times B\lra Q$ such that \[B(b,b')\&\phi(a,b)\&A(a',a)\leq \phi(a',b')\] for any $a,a'\in A$ and $b,b'\in B$. Roughly speaking, a  $\CQ$-distributor $\phi:A\oto B$ is a $\CQ$-relation between $A$ and $B$ that is compatible with the $\CQ$-orders on $A$ and $B$.

It is easy to see that the set $\CQ$-Dist$(A,B)$ of all $\CQ$-distributors from $A$ to $B$ form a complete lattice under the pointwise order.

\begin{exmp}\label{lower set as dist} (Fuzzy lower sets as $\CQ$-distributors)
A \emph{fuzzy lower  set} \cite{LZ06} of  a $\CQ$-ordered set $A$  is a map $\phi:A\lra Q$ such that \[\phi(y)\&A(x,y)\leq\phi(x).\]
It is obvious that $\phi:A\lra Q$ is a fuzzy lower set if and only if $\phi: A\lra(Q, d_R)$ preserves $\CQ$-order.

Dually, a \emph{fuzzy upper  set} \cite{LZ06} of    $A$  is a map $\psi:A\lra Q$ such that \[A(x,y)\&\psi(x)\leq\psi(y).\] It is clear that $\psi:A\lra Q$ is a fuzzy upper set if and only if $\psi: A\lra(Q, d_L)$ preserves $\CQ$-order.

If we write $*$ for the terminal object in the category $\CQ$-{\sf Ord}, namely, $*$ is a $\CQ$-ordered set with only one element, then for each fuzzy lower set $\phi:A\lra Q$ of $A$, the map \[\phi\urcorner:A\times\{*\}\lra Q,\quad \phi\urcorner(x,*)=\phi(x)\] is a $\CQ$-distributor $\phi\urcorner:A\oto *$. This  establishes a bijection between fuzzy lower  sets of $A$ and $\CQ$-distributors from  $A$ to $*$.  For each fuzzy upper set $\psi$ of $A$, the map \[\ulcorner\psi: \{*\}\times A\lra Q,\quad \ulcorner\psi(*,x)=\psi(x)\] is a  $\CQ$-distributor $\psi:*\oto A$. This  establishes a bijection between fuzzy upper  sets of $A$ and $\CQ$-distributors from  $*$ to $A$.
\end{exmp}

\begin{lem}Let  $\phi$ be a fuzzy lower set (fuzzy upper set, resp.) of a $\CQ$-ordered set $A$,    $p\in Q$. \begin{enumerate}[(1)] \item   Both $p\&\phi$ and $p\ra \phi$ are  fuzzy lower sets (fuzzy upper sets, resp.) of $A$.
\item $\phi\ra p$ is a fuzzy upper set (fuzzy lower set, resp.) of $A$ and
    $\phi=\bw_{q\in Q}(\phi\ra q)\ra q)$. \end{enumerate} \end{lem}

Let $\CP A$ denote the set of fuzzy lower sets of $A$ endowed with the fuzzy inclusion order.   Explicitly, elements in $\CP A$ are $\CQ$-order preserving maps $A\lra(Q, d_R)$, and \[\CP A(\phi_1,\phi_2)=\sub_A(\phi_1,\phi_2)=\bw_{x\in A}(\phi_1(x)\ra\phi_2(x)).\]
Dually, let $\CP^\dag A$ denote the set of fuzzy upper sets of $A$ endowed with the \emph{converse} fuzzy inclusion order. Explicitly,  elements in $\CP^\dag A$ are $\CQ$-order preserving maps $A\lra(Q, d_L)$, and \[\CP^\dag A(\psi_1,\psi_2)=\sub_A(\psi_2,\psi_1)=\bw_{x\in A}(\psi_2(x)\ra\psi_1(x)).\] It is clear that $(\CP^\dag A)^{\rm op}=\CP (A^{\rm op})$ \cite{St05}.

For each $a\in A$, $A(-,a)$ is a fuzzy lower set of $A$.  Moreover,
\[\CP A(A(-,a),\phi)=\phi(a)\] for all $a\in A$ and $\phi\in\CP A$. This fact is indeed a special case of the Yoneda lemma in enriched category theory. The Yoneda lemma
ensures that the assignment $a\mapsto A(-,a)$ defines an embedding $\y: A\lra\CP A$,  known as the  Yoneda embedding.

The correspondence $A\mapsto\CP A$ gives rise to a functor $\CP:\CQ$-${\sf Ord}\lra\CQ$-{\sf Ord} that sends a $\CQ$-order preserving map  $f: A\lra B$ to $\CP f=f^\ra:\CP A\lra\CP B$, where \[f^\ra(\phi)(y)=\bv_{x\in A}\phi(x)\&B(y,f(x)).\] Moreover, $f^\ra:\CP A\lra\CP B$ has a right adjoint given by $f^\la:\CP B\lra\CP A$, where $f^\la(\psi) = \psi\circ f$. This means for all $\phi\in\CP A$ and $\psi\in\CP B$, \begin{equation} \label{kan adjunction} \sub_B(f^\ra(\phi),\psi) = \sub_A(\phi,f^\la(\psi)).\end{equation}
The adjunction $f^\ra\dashv f^\la$ is a special case of the enriched Kan extension in category theory \cite{Kelly,Lawvere73}.

For $\CQ$-distributors $\phi:A\oto B$ and $\psi: B\oto C$, the composite $\psi\circ\phi: A\oto C$ is given by $$(\psi\circ\phi)(a,c)=\bv_{b\in B}\psi(b,c)\&\phi(a,b).$$

 It is clear that ($\CQ$-Dist$(*,*),\circ)$ is a quantale and   is isomorphic to  $\CQ=(Q,\&)$.  In this paper, we identify ($\CQ$-Dist$(*,*),\circ)$ with $\CQ$.

For a  fuzzy lower set $\phi:A\lra Q$ and a  fuzzy upper set $\psi:A\lra Q$    of a  $\CQ$-ordered set $A$,  the tensor product \[\phi\otimes\psi\] is defined as the composite of $\CQ$-distributors \[\phi\urcorner\circ\ulcorner\psi: *\oto A\oto*.\] Explicitly, $\phi \otimes \psi$ is an element of the quantale  $\CQ$ given by $\phi \otimes \psi=\bv_{x\in A}\phi(x)\&\psi(x)$. Intuitively, the value $\phi\otimes\psi$ measures the degree that the fuzzy lower set $\phi$ intersects with the fuzzy upper set $\psi$.

The correspondence \[(\psi,\phi)\mapsto \phi \otimes \psi\] defines a $\CQ$-distributor \[ \otimes: \CP^\dag A\oto\CP A.\]  In particular, for each fuzzy upper set $\psi$ of $A$, the correspondence $\phi\mapsto \phi\otimes \psi$ defines a  fuzzy upper set  of $\CP A$: \begin{equation}\label{compositon as distributor}{-\otimes \psi}: \CP A\lra Q. \end{equation}

The following  lemma  exhibits a close relationship between the $\CQ$-distributor $\otimes: \CP^\dag A\oto\CP A$ (intersection degree) and the fuzzy inclusion order (subset degree).

 \begin{lem}\label{tensor via sub} Let $A$ be a $\CQ$-ordered set. \begin{enumerate}[(1)]\item For each fuzzy lower set $\phi$ and each fuzzy upper set $\psi$ of $A$, \[\phi\otimes\psi=\bw_{p\in Q}(\sub_A(\phi,\psi\ra p)\ra p).\] In particular, if $\CQ$ satisfies the law of double negation, then $\phi\otimes\psi
= \neg(\sub_A(\phi, \neg\psi))$. \item  For any fuzzy lower sets $\phi_1,\phi_2$ of $A$, \[\sub_A(\phi_1,\phi_2)=\bw_{p\in Q}(\phi_1\otimes(\phi_2\ra p) \ra p).\]  In particular, if $\CQ$ satisfies the law of double negation, then $\sub_A(\phi_1, \phi_2) 
= \neg(\phi_1\otimes(\neg\phi_2))$. \end{enumerate} \end{lem}
  \begin{proof}(1) By Proposition \ref{2.1}(7), it holds that \begin{align*}\phi\otimes\psi &= \bv_{x\in A}\phi(x)\&\psi(x)\\ &= \bw_{p\in Q}\Big[\Big(\Big(\bv_{x\in A}\phi(x)\&\psi(x)\Big)\ra p\Big) \ra p\Big] \\ &= \bw_{p\in Q}\Big[\bw_{x\in A}(\phi(x)\ra(\psi(x)\ra p))\ra p\Big]  \\ &= \bw_{p\in Q}(\sub_A(\phi,\psi\ra p)\ra p). \end{align*}
(2) The proof is similar, so, we omit it here.  \end{proof}

A supremum of a fuzzy lower set $\phi$ of a $\CQ$-ordered set $A$ is an element of $A$, say $\sup\phi$, such that \[A(\sup\phi,x)=\sub_A(\phi,\y(x))\] for all $x\in A$. It is clear that, up to isomorphism, every fuzzy lower set has at most one supremum. So, we'll speak of \emph{the supremum of a fuzzy lower set}. A $\CQ$-order preserving map $f:A\lra B$ preserves the supremum of a fuzzy lower set $\phi$ of $A$ if,  whenever $\sup\phi$ exists, $f(\sup\phi)$ is a supremum  of $f^\ra(\phi)$. It is well-known that left adjoints preserve suprema.

\begin{exmp}\cite{St05} \label{PA is complete} Let $A$ be a $\CQ$-ordered set. Then every fuzzy lower set of $\CP A$ has a supremum. Actually, for each fuzzy lower set $\Lambda$ of $\CP A$, $\sup\Lambda=\bv_{\phi\in\CP A}\Lambda(\phi)\&\phi$.  \end{exmp}

\begin{exmp}(Intersection degree as supremum) \label{intersection as sup} For each fuzzy lower set  $\phi$ and each fuzzy upper set $\psi$   of a $\CQ$-ordered set $A$, the intersection degree of $\phi$ with $\psi$ is the supremum of $\psi^\ra(\phi)$ in $(Q,d_L)$ (recall that $\psi:A\lra(Q,d_L)$ is a $\CQ$-order preserving map), i.e., $\phi\otimes\psi=\sup\psi^\ra(\phi)$.  This is because for all $q\in Q$, \begin{align*}\sub_Q(\psi^\ra(\phi),d_L(-,q)) &=\sub_A(\phi,d_L(\psi(-), q))\\ &=\bw_{x\in A}(\phi(x)\ra(\psi(x)\ra q)) \\ &= d_L\Big(\bv_{x\in A}\phi(x)\&\psi(x), q\Big)\\ &=d_L(\phi\otimes\psi,q).  \end{align*}

In particular, letting $\psi$   be the identity map on  $(Q,d_L)$ one obtains that for each fuzzy lower set $\phi$ of $(Q,d_L)$, $\sup\phi=\bv_{q\in Q}q\&\phi(q)$. \end{exmp}

\begin{exmp}(Inclusion degree as supremum) \label{inclusion as sup} For any fuzzy lower sets $\phi,\lam$  of a $\CQ$-ordered set $A$,  the inclusion degree $\sub_A(\phi,\lam)$ is the supremum of  $\lam^\ra(\phi)$ in $(Q,d_R)$ (recall that $\lam:A\lra(Q,d_R)$ is a $\CQ$-order preserving map), i.e., $\sub_A(\phi,\lam)=\sup\lam^\ra(\phi)$.  This is because for all $q\in Q$, \begin{align*}\sub_Q(\lam^\ra(\phi),d_R(-,q)) &=\sub_A(\phi,d_R(\lam(-), q))\\ &=\bw_{x\in A}(\phi(x)\ra(q\ra\lam(x))) \\ &= \bw_{x\in A}(q\ra(\phi(x)\ra\lam(x)))\\ &=d_R(\sub_A(\phi,\lam),q).  \end{align*}

In particular, letting $\lam$   be the identity map on  $(Q,d_R)$ one obtains that for each fuzzy lower set $\phi$ of $(Q,d_R)$,   the supremum of $\phi$ in $(Q,d_R)$ is given by $\bw_{q\in Q}(\phi(q)\ra q)$.  \end{exmp}

\section{Forward Cauchy ideals, flat ideals and irreducible ideals}

A net $\{x_i\}$ in a $\CQ$-ordered set $A$ is forward Cauchy \cite{Wagner97} if \[\bv_i\bw_{i\leq j\leq k}A(x_j,x_k)=1.\]
Forward Cauchy nets are clearly a $\CQ$-analogue of eventually monotone nets in partially ordered sets.  A Yoneda limit (a.k.a liminf) \cite{Wagner97} of a forward Cauchy net $\{x_i\}$ in $A$ is an element $a$ in $A$ such that \[A(a,y)=\bv_i\bw_{i\leq j}A(x_j,y) \] for all $y\in A$. It is clear that Yoneda limit is a $\CQ$-version of \emph{least eventual upper bound}.   Yoneda limits of a forward Cauchy net, if exist, are unique up to isomorphism.

\begin{lem}\label{yoneda limit in Q} If $\{a_i\}$ is a forward Cauchy net in $(Q,d_L)$, then $\bv_i\bw_{j\geq i}a_j$ is a Yoneda limit of $\{a_i\}$ and \[\bv_i\bw_{j\geq i}a_j=\bw_i\bv_{j\geq i}a_j.\] \end{lem}

\begin{proof}The first half is   \cite[Proposition 2.30]{Wagner97}. It remains to check the equality  \[\bv_i\bw_{j\geq i}a_j=\bw_i\bv_{j\geq i}a_j.\]

 Since $\bv_i\bw_{j\geq i}a_j$ is a Yoneda limit of $\{a_i\}$, it follows that for all $x\in Q$, \begin{align*} \Big(\bv_i\bw_{j\geq i}a_j\Big)\ra x& = d_L\Big(\bv_i\bw_{j\geq i}a_j,x\Big)\\ &=\bv_i\bw_{j\geq i}d_L(x,a_j) \\ &=\bv_i\bw_{j\geq i}(a_j\ra x) \\ &\leq   \Big(\bw_i\bv_{j\geq i}a_j\Big)\ra x. \end{align*} Letting $x= \bv_i\bw_{j\geq i}a_j$ we obtain that  \[\bv_i\bw_{j\geq i}a_j\geq\bw_i\bv_{j\geq i}a_j.\]  The inequality `$\leq$' is trivial, so, the equality is valid. \end{proof}

 \begin{prop} \label{3.9}  {\rm(A special case of \cite[Theorem 3.1] {Wagner97})} For each forward Cauchy net $\{\phi_i\}$ in $\CP A$, the fuzzy lower set $\bv_i\bw_{j\geq i}\phi_j$ is a Yoneda limit of $\{\phi_i\}$.  That is, for each fuzzy lower set $\phi$ of $A$, \[\sub_A\Big(\bv_i\bw_{j\geq i}\phi_j,\phi\Big)
 =\bv_i\bw_{j\geq i}\sub_A(\phi_j,\phi). \] \end{prop}

The following proposition says that every Yoneda limit  of forward Cauchy net $\{x_i\}$ is a  supremum of a fuzzy lower set generated by $\{x_i\}$.

\begin{prop} {\rm (\cite[Lemma 46]{FSW})} \label{yoneda limit as suprema} An element $a$ in a $\CQ$-ordered set $A$ is a Yoneda limit of a forward Cauchy net $\{x_i\}$ if and only if $a$ is a supremum of the fuzzy lower set $\bv_i\bw_{i\leq j}A(-,x_j)$ generated by $\{x_i\}$. \end{prop}

A fuzzy set $\lam:A\lra Q$ is \emph{inhabited}   if $\bv_{a\in A}\lam(a)=1$.  Inhabited fuzzy sets are counterpart of non-empty sets in the fuzzy setting.

\begin{defn}Let $A$ be a $\CQ$-ordered set, $\phi:A\lra Q$ a fuzzy lower set of $A$.   \begin{enumerate} \item[\rm(1)]
 $\phi $   is   a forward Cauchy ideal  if there exists a forward Cauchy net $\{x_i\}$ in $A$ such that \[\phi=\bv_i\bw_{i\leq j}A(-,x_j).\]
 \item[\rm(2)]
 $\phi$ is   a flat ideal if it is inhabited and is flat in the sense that  \[\phi\otimes(\psi_1\wedge\psi_2)= \phi\otimes\psi_1  \wedge\phi\otimes\psi_2   \] for all  fuzzy upper sets  $\psi_1, \psi_2$ of $A$.
 \item[\rm(3)]
 $\phi$  is   an irreducible ideal   if it is inhabited and is irreducible in the sense that \[{\rm sub}_A(\phi, \phi_1\vee\phi_2)={\rm sub}_A(\phi, \phi_1)\vee{\rm sub}_A(\phi, \phi_2)  \] for all  fuzzy lower sets  $\phi_1, \phi_2$ of $A$. \end{enumerate} \end{defn}

\begin{rem}\label{history}  Forward Cauchy ideals, flat ideals  and irreducible ideals in  $\CQ$-ordered sets are all natural extensions of ideals in a partially ordered set.
The study of forward Cauchy ideals dates back to  Wagner \cite{Wagner94,Wagner97}. For more information on forward Cauchy ideals the reader is referred to \cite{FK97,FS02,FSW,HW2011,HW2012,LZ07,ZF05}, besides the works of Wagner. The notion of flat ideals originates in the paper \cite{SV2005} of Vickers in the case that $\CQ$ is Lawvere's quantale $([0,\infty]^{\rm op},+)$, under the name of \emph{flat left module}. It is extended to the general case in \cite{TLZ2014}. It is shown in \cite{TLZ2014} that if the quantale $\CQ=(Q,\&)$ is a frame, i.e., $\&=\wedge$, then a fuzzy lower set $\phi$ of a $\CQ$-ordered set $A$ is flat if and only if   for any $x,y\in A$, \[\phi(x)\wedge\phi(y)\leq\bv_{z\in A}\phi(z)\wedge A(x,z)\wedge A(y,z).\] Hence, in the case that $\CQ=(Q,\&)$ is a frame,  flat ideals in a $\CQ$-ordered set $A$ coincides with  ideals of $A$ in the sense of \cite[Definition 5.1]{LZ07}. Irreducible ideals are introduced in \cite{Zhang18} in the study of sobriety of $\CQ$-cotopological spaces.
\end{rem}

\begin{exmp}\label{principal ideal} For each $a$ in a $\CQ$-ordered set $A$, the fuzzy lower set $A(-,a)$ is a forward Cauchy ideal, a flat ideal  and an irreducible ideal.
\end{exmp}

\begin{defn} (\cite{AK,KS05,LZ07}) \label{class} By a class of weights  we mean a functor
$\Phi: \CQ$-${\sf Ord} \lra\CQ$-{\sf Ord}  such that
\begin{enumerate}[(1)] \item for each $\CQ$-ordered set $A$, $\Phi(A)$ is a subset of $\CP A$ with the $\CQ$-order inherited from $\CP A$;
\item for all $\CQ$-ordered set $A$ and all $a\in A$, $\y(a)\in\Phi(A)$;
\item $\Phi(f)= \CP f =f^\ra$ for every $\CQ$-order preserving map $f: A\lra B$. \end{enumerate} \end{defn}

The second condition ensures that $A$ can be embedded in $\Phi(A)$ via the Yoneda embedding.
We also write $\y$ for the embedding $A\lra\Phi(A)$ if no
confusion will arise.

In category theory, a $\CQ$-distributor of the form $A\oto *$ is called  a \emph{weight} or a  \emph{presheaf}  \cite{KS05,St05}. This accounts for the terminology \emph{class of weights}.

Together with Example \ref{principal ideal} the following conclusion asserts that forward Cauchy ideals, flat ideals and irreducible ideals are all examples of class of weights.
\begin{prop}If $f:A\lra B$ is $\CQ$-order preserving, then for each forward Cauchy ideal (flat ideal,  irreducible ideal, resp.)  $\phi$   of $A$,   $f^\ra(\phi)$ is a forward Cauchy ideal (flat ideal, irreducible ideal, resp.) of $B$. \end{prop}
\begin{proof} We check, for example, that if $\phi$ is irreducible then so is $f^\ra(\phi)$.  For all fuzzy lower sets $\phi_1, \phi_2$ of $B$,  thanks to Equation (\ref{kan adjunction}), we have
\begin{align*}\sub_B(f^\ra(\phi),\phi_1\vee\phi_2) &= \sub_A(\phi,(\phi_1\vee\phi_2)\circ f) \\ &= \sub_A(\phi,\phi_1\circ f)\vee\sub_A(\phi,\phi_2\circ f)\\ &= \sub_B(f^\ra(\phi),\phi_1)\vee\sub_B(f^\ra(\phi),\phi_2), \end{align*}   hence $f^\ra(\phi)$ is irreducible. \end{proof}

In the following, we write $\mathcal{W}$, $\CI$ and $\CF$ for the class of forward Cauchy ideals, irreducible ideals and flat ideals, respectively.

\begin{defn}Let $\Phi$ be a class of weights. A $\CQ$-ordered set $A$  is $\Phi$-complete\footnote{~\emph{$\Phi$-cocomplete} would be a better terminology from the viewpoint of category theory. However, following the tradition in domain theory, we choose \emph{$\Phi$-complete} here.} if  each $\phi\in\Phi(A)$ has a supremum. In particular, a $\CQ$-ordered set  $A$ is \begin{enumerate}[(1)] \item   Yoneda complete  (a.k.a liminf complete) if each forward Cauchy ideal of $A$ has a supremum (which is equivalent to that every forward Cauchy net in $A$ has a Yoneda limit); \item
  irreducible complete  if each irreducible ideal of $A$ has a supremum; \item
 flat complete  if each flat ideal of $A$ has a supremum.  \end{enumerate} \end{defn}

Yoneda complete, irreducible complete, and flat complete are all natural extension of \emph{directed complete} to the fuzzy setting.  In the case that $\CQ=(Q,\&)$ is a frame, based on flat completeness (under the name of \emph{fuzzy directed completeness}), a theory of frame-valued directed complete orders and frame-valued domains have been developed in  \cite{Liu-Zhao,Yao10,Yao16,YS11}.

It is easily seen that $A$ is $\Phi$-complete if and only if  $\y:A\lra\Phi(A)$ has a left adjoint. In this case, the left adjoint of $\y$ sends each $\phi\in\Phi(A)$ to its supremum $\sup\phi$.
A $\CQ$-order preserving  map $f:A\lra B$  is \emph{$\Phi$-cocontinuous} if  for all $\phi\in\Phi(A)$,  $f(\sup\phi)=\sup f^\ra(\phi)$ whenever $\sup\phi$ exists.

This section mainly concerns the relationship among the class $\mathcal{W}$ of forward Cauchy ideals, the class $\CI$ of irreducible ideals, and the class $\CF$ of flat ideals.

Given classes of weights $\Phi$ and $\Psi$, we say that $\Phi$ is a subclass of $\Psi$ if $\Phi(A)\subseteq \Psi(A)$ for each $\CQ$-ordered set $A$. As we shall see, under some mild assumptions, $\mathcal{W}$ is a subclass of both $\CI$ and $\CF$.

A complete lattice $L$ is meet continuous \cite{Gierz2003} if for all $a\in L$ and all directed subset $D$ of $L$,  \[a\wedge \bv D=\bv_{d\in D}(a\wedge d).\] A complete lattice is dually meet continuous if its opposite   is meet continuous.
A quantale $\CQ=(Q,\&)$ is (dually, resp.) meet continuous if the complete lattice $Q$ is (dually, resp.) meet continuous.

 \begin{thm}\label{FC is irreducible} For a dually meet continuous quantale  $\CQ$, every forward Cauchy ideal is  irreducible.  \end{thm}

 \begin{lem}\label{order convergence} If $\{a_i\}$ is a forward Cauchy net in the $\CQ$-ordered set $(Q,d_R)$, then $ \bw_i\bv_{j\geq i}a_j$ is a Yoneda limit of $\{a_i\}$ and \[\bv_i\bw_{j\geq i}a_j=\bw_i\bv_{j\geq i}a_j.\] \end{lem}
 \begin{proof}First, we show that $\bw_i\bv_{j\geq i}a_j$ is a Yoneda limit of $\{a_i\}$. That is, for all $x\in Q$, \[d_R\Big(\bw_i\bv_{j\geq i}a_j,x\Big)=\bv_i\bw_{j\geq i}d_R(a_j,x).\]

On one hand, since $\{a_i\}$ is a forward Cauchy net in $(Q,d_R)$,  \[\bv_i\bw_{i\leq j\leq l}(a_l\ra a_j)=\bv_i\bw_{i\leq j\leq l}d_R(a_j,a_l)=1,\] then
 \[\bv_i\bw_{j\geq i}\Big[\bv_k\bw_{l\geq k}(a_l\ra a_j)\Big]=1,\] hence  \[\bv_i\bw_{j\geq i}\Big[\Big(\bw_k\bv_{l\geq k}a_l\Big)\ra a_j\Big]=1.\] Thus, \begin{align*}d_R\Big(\bw_i\bv_{j\geq i}a_j,x\Big)&= \Big[x\ra\bw_k\bv_{l\geq k}a_l\Big]\&\bv_i\bw_{j\geq i}\Big[\Big(\bw_k\bv_{l\geq k}a_l\Big)\ra a_j)\Big]\\ &\leq  \bv_i\bw_{j\geq i}(x\ra a_j) \\ &=\bv_i\bw_{j\geq i}d_R(a_j,x). \end{align*}

 On the other hand, since for each $i$ we always have \[x\ra \bv_{j\geq i}a_j \geq \bv_k\bw_{l\geq k}(x\ra a_l),\] it follows that \[d_R\Big(\bw_i\bv_{j\geq i}a_j,x\Big)=\bw_i\Big(x\ra\bv_{j\geq i}a_j\Big)\geq \bv_k\bw_{l\geq k}(x\ra a_l)=\bv_i\bw_{j\geq i}d_R(a_j,x) . \]

 Therefore, $\bw_i\bv_{j\geq i}a_j$ is a Yoneda limit of $\{a_i\}$.

 Next, we prove the equality \[\bv_i\bw_{j\geq i}a_j=\bw_i\bv_{j\geq i}a_j.\]

 Since $\bw_i\bv_{j\geq i}a_j$ is a Yoneda limit of $\{a_i\}$ in $(Q,d_R)$, it follows that for all $x\in Q$, \begin{align*} x\ra \bw_i\bv_{j\geq i}a_j& = d_R\Big(\bw_i\bv_{j\geq i}a_j,x\Big)\\ &=\bv_i\bw_{j\geq i}d_R(a_j,x) \\ &=\bv_i\bw_{j\geq i}(x\ra a_j) \\ &\leq x\ra \bv_i\bw_{j\geq i}a_j. \end{align*} Letting $x=1$, we obtain that \[\bw_i\bv_{j\geq i}a_j\leq \bv_i\bw_{j\geq i}a_j.\]

The converse inequality is trivial, so the equality is valid. \end{proof}

Lemma \ref{yoneda limit in Q} and Lemma \ref{order convergence} imply that every forward Cauchy net in the $\CQ$-ordered sets $(Q,d_L)$ and $(Q,d_R)$ is order convergent. But,  $(Q,d_L)$ and $(Q,d_R)$ may have different forward Cauchy nets. For example, the sequence  $\{n\}$ is forward Cauchy in $([0,\infty],d_R)$ but not in $([0,\infty],d_L)$.

 \begin{proof}[Proof of Theorem \ref{FC is irreducible}] Let $\{x_i\}$ be a forward Cauchy net in a $\CQ$-ordered set $A$ and $\varphi=\bv_i\bw_{j\geq i}A(-,x_j)$. We  show that $\varphi$ is an irreducible ideal.

 \textbf{Step 1}. $\varphi$ is inhabited. This is easy since \[\bv_{x\in A}\varphi(x)\geq\bv_{i}\varphi(x_i)=\bv_{i}\bv_j\bw_{k\geq j}A(x_i,x_k)\geq \bv_{i} \bw_{k\geq i}A(x_i,x_k)=1.\]

\textbf{Step 2}. For each fuzzy lower set $\phi$ of $A$, \[\sub_A(\varphi,\phi)= \bw_i\bv_{j\geq i}\phi(x_j).\]

Since $\phi$ is a fuzzy lower set, $\{\phi(x_j)\}$ is a forward Cauchy net in $(Q, d_R)$.
Then, \begin{align*}\sub_A(\varphi,\phi)&= \sub_A\Big(\bv_i\bw_{j\geq i}A(-,x_j),\phi\Big) \\ &
 =\bv_i\bw_{j\geq i}\sub_A(A(-,x_j),\phi) & \text{(Proposition \ref{3.9})}\\ &= \bv_i\bw_{j\geq i}\phi(x_j)& \text{(Yoneda lemma)}\\ &= \bw_i\bv_{j\geq i}\phi(x_j).& \text{(Lemma \ref{order convergence})} \end{align*}

\textbf{Step 3}. For all  fuzzy lower sets  $\phi_1, \phi_2$ of $A$, ${\rm sub}_A(\varphi, \phi_1\vee\phi_2)={\rm sub}_A(\varphi, \phi_1)\vee{\rm sub}_A(\varphi, \phi_2)$.

This is easy since   \begin{align*}{\rm sub}_A(\varphi, \phi_1)\vee{\rm sub}_A(\varphi, \phi_2)&= \bw_i\bv_{j\geq i}\phi_1(x_j)\vee \bw_i\bv_{j\geq i}\phi_2(x_j)& \text{(Step 2)}\\ & = \bw_i\bv_{j\geq i}(\phi_1(x_j)\vee\phi_2(x_j)) & \text{($\CQ$ is dually meet continuous)}\\&= {\rm sub}_A(\varphi, \phi_1\vee\phi_2).& \text{(Step 2)}\end{align*}

The proof is completed.
\end{proof}

Interestingly, the dual meet continuity of $\CQ$ is also necessary for Theorem \ref{FC is irreducible}.

\begin{prop}\label{DMC is necessary} If all forward Cauchy ideals are irreducible, then the quantale $\CQ$ is dually meet continuous. \end{prop}

\begin{proof}We show that for each $a\in Q$ and each  filtered set $F$ in $Q$, \[a\vee\bw_{x\in F} x=\bw_{x\in F}(a\vee x).\]

Consider the fuzzy lower set $\phi =\bv_{x\in F}d_R(-,x)$ of the $\CQ$-ordered set $(Q, d_R)$. Since \[\phi  =\bv_{x\in F}\bw_{y\in F,y\leq x}d_R(-,y),\] it follows that $\phi$ is a forward Cauchy ideal, hence an irreducible ideal by assumption.

Since both the identity map ${\rm id}_Q$ on $Q$ and the constant map $\underline{a}:Q\lra Q$ with value $a$  are fuzzy lower sets of  $(Q, d_R)$, then \begin{align*}a\vee\bw_{x\in F} x&=\sub_Q(\phi,\underline{a})\vee\sub_Q(\phi,{\rm id}_Q) \\ &=\sub_Q(\phi,\underline{a}\vee{\rm id}_Q) \\ &= \bw_{x\in F}(a\vee x).\end{align*} This finishes the proof. \end{proof}

Irreducible ideals need not be forward Cauchy  in general. Let $\CQ=\{0,a,b,1\}$ be the Boolean algebra with four elements. Assume that $A$ is the $\CQ$-ordered set with  points $x,y$  and \[A(x,x)=A(y,y)=1, \quad A(x,y)=A(y,x)=0.\] Then  the map $\phi$, given by $\phi(x)=a$ and $\phi(y)=b$, is an irreducible ideal in $A$. But, $\phi$ cannot be generated by any forward Cauchy net in $A$. This example is essentially \cite[Note 3.12]{Zhang18}.

The following conclusion is very useful in the theory of fuzzy orders based on left continuous t-norms, it says every irreducible ideal is forward Cauchy in this case.

\begin{thm}\label{G is FC} If the quantale $\CQ$ is the unit interval equipped with a left continuous t-norm $\&$, then irreducible ideals coincide with forward Cauchy ideals. \end{thm}

\begin{proof} By Theorem \ref{FC is irreducible}, we only need to prove that every   irreducible ideal $\phi $ of a $\CQ$-ordered set $A$ is   forward Cauchy.
 	
Let \[{\rm C}\phi = \{(x,r)\in X\times[0,1)\mid \phi(x)>r\}.\] Define a relation $\sqsubseteq$ on ${\rm C}\phi $ by \[ (x,r)\sqsubseteq(y,s)\iff A(x,y)\ra r \leq s. \]
We claim that
$({\rm C}\phi,\sqsubseteq)$  is a directed set. Before proving this, we note that if $(x,r)\sqsubseteq(y,s)$ then $r<A(x,y) $ and $r\leq s$.  That $\sqsubseteq$ is reflexive and transitive is easy, it remains to check that it is directed. For any $ (x,r),(y,s)\in {\rm C}\phi$,  consider the fuzzy lower sets $\psi_1=A(x,-)\ra r$ and $ \psi_2=A(y,-)\ra s$. Since $\phi $ is an irreducible ideal,   \begin{align*}
 		\sub_X(\phi,\psi_1\vee\psi_2)
 &=\sub_X(\phi,A(x,-)\ra r)\vee \sub_X(\phi,A(y,-)\ra s)                            	\\
 &= \sub_X(A(x,-),\phi\ra r)\vee \sub_X(A(y,-),\phi\ra s)\\ &=(\phi(x)\ra r)\vee(\phi(y)\ra s).
 	\end{align*}
Since $(\phi(x)\ra r)\vee(\phi(y)\ra s)<1$,   there exists some $z$ such that \[ \phi(z)\ra\big[(A(x,z)\ra r)\vee(A(y,z)\ra s)\big] <1. \] Let $ t=(A(x,z)\ra r)\vee(A(y,z)\ra s)$, then  $(z,t)\in {\rm C}\phi$ and $(x,r)\sqsubseteq(z,t), (y,s)\sqsubseteq(z,t)$. Hence $({\rm C}\phi,\sqsubseteq)$  is a directed set.	

From now on, we also write an element  in ${\rm C}\phi$ as a pair $(x_i,r_i)$. Define a net \[\mathfrak{x}:{\rm C}\phi\lra A\] by $\mathfrak{x}(x_i,r_i)=x_i.$
We prove in two steps that $\mathfrak{x}$ is a forward Cauchy net and	it generates $\phi$, hence $\phi$ is a forward Cauchy ideal.
 	
\textbf{Step 1}. $\mathfrak{x}$ is   forward Cauchy.

Let $t<1$. Since $\phi$ is inhabited, there is some $(x_i,r_i)\in {\rm C}\phi$ such that $t\leq r_i$. Then   $A(x_j,x_k)\ra r_j\leq r_k<1$ whenever $(x_k,r_k)\sqsupseteq(x_j, r_j)\sqsupseteq(x_i,r_i)$,   hence \[A(x_j,x_k)> r_j\geq r_i\geq t.\] By arbitrariness of $t$ we obtain that $\mathfrak{x}$ is forward Cauchy.

\textbf{Step 2}. $\phi$ is generated by   $\mathfrak{x}$, i.e., \[ \phi(x)=\bv_{(x_i,r_i)}\bw_{(x_j,r_j) \sqsupseteq (x_i,r_i)}A(x,x_{j}) \] for all $x\in A$.
 	
Take $x\in A$ and $r<\phi(x)$. For all $(x_j,r_j)\in  {\rm C}\phi$, if $(x,r)\sqsubseteq(x_j,r_j)$, then $A(x,x_j)> r$, hence, by arbitrariness of $r$, \[ \phi(x)\leq\bv_{r<\phi(x)}\bw_{(x_j,r_j) \sqsupseteq (x,r)}A(x,x_{j})\leq\bv_{(x_i, r_i)}\bw_{(x_j,r_j) \sqsupseteq (x_i,r_i)}A(x,x_{j}). \]
 	
For the converse inequality, we show that for each $(x_i,r_i)\in{\rm C}\phi$,    \[ \phi(x)\geq \bw_{(x_j,r_j) \sqsupseteq (x_i,r_i)}A(x,x_j). \]
Let $t$ be an arbitrary number that is strictly smaller than \[\bw_{(x_j,r_j) \sqsupseteq (x_i,r_i)}A(x,x_j).\]	Since $\& $ is left continuous and $\phi$ is inhabited, there is some $(x_k,r_k) \in {\rm C}\phi$ such that \[ t\leq r_k\&\bw_{(x_j,r_j) \sqsupseteq (x_i,r_i)}A(x,x_j). \] Take some $(x_l,r_l)\in {\rm C}\phi$ such that $(x_i,r_i), (x_k,r_k)\sqsubseteq(x_l,r_l)$. Then \begin{align*}
 	\phi(x) \geq\phi(x_l)\&A(x,x_l)
 	 \geq r_k\&  A(x,x_l)\geq t.\end{align*} Therefore, by   arbitrariness of $t$,   \[\phi(x)\geq \bw_{(x_j,r_j) \sqsupseteq (x_i,r_i)}A(x,x_j).  \]
  The proof is completed. \end{proof}

A slight improvement of the argument shows that  the above theorem is valid for all linearly ordered quantales. That is, if $\CQ$ is a linearly ordered quantale, then irreducible ideals coincide with forward Cauchy ideals.

As an application,  the following corollary characterizes, for the quantale $\CQ=([0,1],\&)$ with $\&$ being a left continuous t-norm, the irreducible ideals in the $\CQ$-ordered sets $([0,1],d_L)$ and $([0,1],d_R)$.

\begin{cor} \label{irreducible ideals in [0,1]} Let $\&$ be a left continuous t-norm and $\CQ=([0,1],\&)$.  \begin{enumerate}[(1)]\item A fuzzy lower set $\phi$  of   the $\CQ$-ordered set $([0,1],d_L)$ is an irreducible ideal if and only if either  $\phi(x)=x \rightarrow a$ for some $a\in [0,1]$ or $\phi(x)=\bv_{b<a}(x\ra b)$ for some $ a>0$.
\item A fuzzy lower set $\psi$ of the $\CQ$-ordered set $([0,1],d_R)$  is an irreducible ideal if and only if either  $\psi(x)=a \rightarrow x$ for some $a\in [0,1]$ or $\psi(x)=\bv_{b>a}(b\ra x)$ for some $ a<1$. \end{enumerate} \end{cor}
\begin{proof}(1) Sufficiency is easy since the fuzzy lower set $\phi(x)=\bv_{b<a}(x\ra b)$ is generated by the forward Cauchy sequence $\{a-1/n\}$. As for  necessity, suppose that $\phi$ is an irreducible ideal of $([0,1],d_L)$. Then there is a forward Cauchy net $\{x_i\}$ in $([0,1],d_L)$ such that \[\phi(x)=\bv_i\bw_{j\geq i}(x\ra x_j).\] Let $a= \bv_i\bw_{j\geq i}x_j$. Then \[\phi(x)=\bv_i\bw_{j\geq i}(x\ra x_j)=\bv_i\Big(x\ra \bw_{j\geq i}x_j\Big), \] hence either $\phi(x)=x \rightarrow a$   or $\phi(x)=\bv_{b<a}(x\ra b)$.

(2) Similar to (1). \end{proof}

\begin{thm}\label{FC is flat} For a meet continuous quantale  $\CQ$, every forward Cauchy ideal is   flat.   \end{thm}
\begin{proof}We only need to show that if $\{x_i\}$ is a forward Cauchy net in a $\CQ$-ordered set $A$ and \[\varphi=\bv_i\bw_{j\geq i}A(-,x_j),\] then $\varphi$ is flat. We do this in two steps.

\textbf{Step 1}. For each fuzzy upper set $\psi$ of $A$, \[\varphi\otimes\psi=\bv_i\bw_{j\geq i}\psi(x_j).\]

Since $\psi$ is a fuzzy upper set, $\{\psi(x_i)\}$ is a forward Cauchy net in $(Q,d_L)$, hence \begin{align*}\varphi\otimes\psi &= \bw_{p\in Q}(\sub_X(\varphi,\psi\ra p)\ra p) &(\text{Lemma \ref{tensor via sub}})\\ &=  \bw_{p\in Q}\Big(\sub_X\Big(\bv_i\bw_{j\geq i}A(-,x_j),\psi\ra p\Big)\ra p\Big) \\ &=  \bw_{p\in Q}\Big(\Big(\bv_i\bw_{j\geq i}\sub_A \big(A(-,x_j),\psi\ra p \big)\Big)\ra p\Big)&(\text{Proposition \ref{3.9}}) \\ &=  \bw_{p\in Q}\Big(\Big(\bv_i\bw_{j\geq i}(\psi(x_j)\ra p)\Big)\ra p\Big) &(\text{Yoneda lemma})\\ &=  \bw_{p\in Q}\Big(\Big(\bv_i\bw_{j\geq i}\psi(x_j)\ra p\Big)\ra p\Big) &(\text{Lemma \ref{yoneda limit in Q}})\\ &= \bv_i\bw_{j\geq i}\psi(x_j). \end{align*}

\textbf{Step 2}. $\varphi$ is flat. For any fuzzy upper sets $\psi_1$ and $\psi_2$, \begin{align*}\varphi\otimes \psi_1\wedge\phi\otimes\psi_2 &=\bv_i\bw_{j\geq i} \psi_1(x_j) \wedge\bv_i\bw_{j\geq i} \psi_2(x_j) &(\text{Step 1})\\ &   = \bv_i\bw_{j\geq i}(\psi_1(x_j)\wedge\psi_2(x_j)) & (\CQ~\text{is meet continuous})\\ &  = \varphi\otimes(\psi_1\wedge\psi_2). &(\text{Step 1}) \end{align*}

The proof is completed.
\end{proof}

Similar to Proposition \ref{DMC is necessary}, it can be shown that the meet continuity of $\CQ$  is also necessary in Theorem \ref{FC is flat}.

\begin{prop}\label{MC is necessary} If all forward Cauchy ideals are flat, then the quantale $\CQ$ is meet continuous. \end{prop}

\begin{exmp}
Consider the quantale $\CQ=([0,1],\wedge)$ and the $\CQ$-ordered set $ ([0,1],d_L)$.
By linearity of $[0,1]$, every  fuzzy lower set $\phi$ of $([0,1],d_L)$ satisfies   \[\phi(x)\wedge \phi(y)\leq \bigvee\limits_{z\in[0,1]}\phi(z)\wedge d_L(x, z)\wedge d_L(y, z),  \] hence every inhabited fuzzy lower set $\phi$  of $([0,1],d_L)$ is a flat ideal by Remark \ref{history}.   In particular, the map $\phi:[0,1]\lra[0,1]$, given by \[\phi(x)=\begin{cases}1-x, & x\leq 1/2, \\ 1/2, &x>1/2,\end{cases} \] is a flat ideal of $([0,1],d_L)$. But, it is not  irreducible  by Corollary \ref{irreducible ideals in [0,1]}, hence not forward Cauchy by Theorem \ref{FC is irreducible}.  \end{exmp}

In the case that $\CQ$ is the interval $[0,1]$ equipped with a continuous t-norm, we are able to present a sufficient and necessary condition for flat ideals to be forward Cauchy.

\begin{thm}\label{main} Let $\CQ$ be the unit interval equipped with a continuous t-norm $\&$. The following are equivalent: \begin{enumerate}[(1)] \item $\&$ is Archimedean, i.e., $\&$ has no non-trivial idempotent elements.
\item  Every flat ideal is forward Cauchy. \item Every flat ideal is  irreducible.
\end{enumerate} \end{thm}

We prove a lemma first. A quantale $\CQ=(Q,\&)$ is called \emph{divisible} \cite{Ha98,Ho95} if \[x\&(x\ra y)=x\wedge y\] for all $x,y\in Q$. It is known that the underlying lattice of a divisible quantale is a frame, hence a distributive lattice, see e.g. \cite{Ho95}. Let $\CQ$ be a divisible quantale and $b$ an idempotent element. Then for all $x\in Q$, \[b\wedge x=b\&(b\ra x)=b\&b\&(b\ra x)\leq b\&x\leq b\wedge x,\] hence $b\wedge x=b\&x$.

\begin{lem}\label{3.11} Suppose $\CQ=(Q,\&)$ is a divisible quantale, $b$ is an idempotent element of $\&$.   Then for each $a\in Q$, the map  $\phi(x)=b\vee(x\ra a)$  is a flat ideal of the $\CQ$-ordered set $(Q,d_L)$.\end{lem}

\begin{proof}   It is clear that $\phi$ is a fuzzy lower set of $(Q,d_L)$ and $\bv_{x\in Q}\phi(x)=1$. It remains to check that $\phi\otimes(\psi_1\wedge\psi_2)= (\phi\otimes\psi_1)\wedge(\phi\otimes\psi_2)$ for all upper sets   $\psi_1,\psi_2 $  of $(Q,d_L)$.

Because for $i=1,2$, \begin{align*} \phi\otimes \psi_i  &= \bv_{x\in Q}((b\&\psi_i(x))\vee((x\ra a)\&\psi_i(x)))\\ &= (b\wedge\psi_i(1))\vee \psi_i(a) &(\text{$b$ is idempotent})\\ & = (b\vee\psi_i(a))\wedge\psi_i(1), &(\text{$Q$ is distributive}) \end{align*}  therefore \begin{align*}(\phi\otimes\psi_1)\wedge(\phi\otimes\psi_2)&= (b\vee\psi_1(a))\wedge\psi_1(1)\wedge(b\vee\psi_2(a))\wedge\psi_2(1)\\ &= (b\vee(\psi_1(a)\wedge\psi_2(a))\wedge (\psi_1(1)\wedge\psi_2(1))\\ &= \phi\otimes(\psi_1\wedge\psi_2).    \end{align*} This completes the proof. \end{proof}

\begin{proof}[Proof of Theorem \ref{main}] $(1)\Rightarrow(2)$ This is contained in \cite[Proposition 7.9]{LLZ17}. If $\&$ is isomorphic to the product t-norm, an equivalent version of this implication can also be found in Vickers \cite{SV2005}.

$(2)\Rightarrow(3)$ This follows immediately from Theorem \ref{FC is irreducible}.

$(3)\Rightarrow(1)$  Suppose $b$ is a  non-trivial idempotent element of $\&$. Take some $a\in (0,b)$. Since  $[0,1]$ together with a continuous t-norm is a divisible quantale \cite{Belo02,Ha98}, it follows from Lemma \ref{3.11} that $\phi(x)=b\vee(x\ra a)$ is a flat ideal of $([0,1],d_L)$. But, $\phi$ is not irreducible, because neither $\phi(x)\leq b$ for all $x$ nor $\phi(x)\leq x\ra a$ for all $x$, a contradiction. \end{proof}

Finally, we discuss the relationship between irreducible ideals and flat ideals.

A quantale $\CQ$ is prelinear if  $(p\ra q)\vee(q\ra p)=1$ for all $p,q\in Q$. It is known that $\CQ$ is prelinear if and only if $(p\wedge q)\ra r= (p\ra r)\vee(q\ra r)$ for all $p,q,r\in Q$, see e.g. \cite{Belo02}.

\begin{prop}\label{irr is flat}  If $\CQ$ is prelinear, then every irreducible ideal is a flat ideal. \end{prop}

\begin{proof}Assume that $\phi$ is an irreducible ideal. Then for all fuzzy upper sets $\psi_1,\psi_2$, \begin{align*}\phi\otimes(\psi_1\wedge\psi_2) &=  \bw_{p\in Q} (\sub_A (\phi,(\psi_1\wedge\psi_2)\ra p )\ra p )\\ &=  \bw_{p\in Q} (\sub_A (\phi,(\psi_1\ra p)\vee(\psi_2 \ra p) )\ra p )\\ &=  \bw_{p\in Q}\Big( (\sub_A (\phi, \psi_1\ra p )\ra p )\wedge (\sub_A (\phi, \psi_2 \ra p  )\ra p )\Big)\\ &= (\phi\otimes\psi_1)\wedge(\phi\otimes\psi_2), \end{align*} hence $\phi$ is flat. \end{proof}

\begin{prop} If $\CQ$ satisfies the law of double negation then flat ideals coincide with irreducible ideals.\end{prop}
 \begin{proof}If $\CQ$  satisfies the law of double negation, then, by Lemma \ref{tensor via sub}, for all fuzzy lower sets $\phi,\varphi$ and fuzzy upper set $\psi$, \begin{equation*}\sub_A(\phi, \varphi)
= \phi\otimes(\varphi\ra0)\ra0 \end{equation*} and \begin{equation*}\phi\otimes\psi
= \sub_A(\phi, \psi\ra0)\ra0. \end{equation*} The conclusion follows easily from these equations. \end{proof}

\begin{cor}\label{irreducible weights in [0,1] NM} Let $\CQ$ be the unit interval equipped with a left continuous t-norm $\&$. If $\CQ$ satisfies the law of double negation, then for each fuzzy lower set $\phi$ of a $\CQ$-ordered set, the following are equivalent: \begin{enumerate}[\rm(1)] \item $\phi$ is a forward Cauchy ideal. \item $\phi$ is an irreducible ideal. \item $\phi$ is   a flat ideal.  \end{enumerate} \end{cor}

\section{Saturatedness}
Let $\Phi$ be a class of weights. A $\CQ$-ordered set $A$ is \emph{$\Phi$-continuous} if it is $\Phi$-complete and the left adjoint $\sup:\Phi(A)\lra A$ of $\y:A\lra\Phi(A)$ has a left adjoint. This kind of postulation is standard in order theory \cite{Wood}. In the case that $\Phi=\CP$ (the largest class of weights), $\Phi$-continuous $\CQ$-ordered sets are the completely distributive (or, totally continuous) $\CQ$-categories in \cite{PZ15,St07}.   We write $\Phi$-{\sf Cont} for the category of $\Phi$-continuous $\CQ$-ordered sets and $\Phi$-cocontinuous maps. The category $\Phi$-{\sf Cont} is the subject of fuzzy domain theory. So, a natural question is whether   such $\CQ$-ordered sets exist. As we will see, saturatedness of $\Phi$ guarantees that there exist enough such things.

A class of weights $\Phi$ is  {\it saturated}
\cite{KS05,LZ07} if for all $\CQ$-ordered set $A$  and $\Lambda\in\Phi(\Phi(A))$, \[\bv_{\phi\in\Phi(A)}\Lambda(\phi)\&\phi\in\Phi(A).\] A category-minded reader will recognize soon that a saturated class of weights is an example of KZ-monads \cite{Kock,Zo76}.

\begin{thm}Let $\Phi$ be a saturated class of weights. \begin{enumerate}[(1)] \item For each $\CQ$-ordered set $A$,   $\Phi(A)$ is $\Phi$-continuous. \item For each $\CQ$-order preserving map $f:A\lra B$, $\Phi(f)$ is $\Phi$-cocontinuous. \item  The functor $\Phi:\CQ$-${\sf Ord}\lra\Phi$-{\sf Cont}, which sends each $\CQ$-order preserving map $f$ to $\Phi(f)$, is a left adjoint of the forgetful functor $\Phi$-${\sf Cont}\lra\CQ$-{\sf Ord}.
\end{enumerate} \end{thm}

\begin{proof}(1) For each   $\Lambda\in\Phi(\Phi(A))$, since $\Phi$ is saturated,  $\bv_{\phi\in\Phi(A)}\Lambda(\phi)\&\phi$ belongs to $\Phi(A)$. It is easy to verify that for all $\psi\in\Phi(A)$, \begin{align*}\CP \Phi(A) (\Lambda, \Phi(A)(-,\psi))
&= \Phi(A)\Big(\bv_{\phi\in\Phi(A)}\Lambda(\phi)\&\phi, \psi\Big) \end{align*}  hence $\bv_{\phi\in\Phi(A)}\Lambda(\phi)\&\phi$ is a supremum of $\Lambda$ in $\Phi(A)$, i.e., \[{\sup}_{\Phi(A)} \Lambda =\bv_{\phi\in\Phi(A)}\Lambda(\phi)\&\phi.\] This shows that $\Phi(A)$ is $\Phi$-complete.

Next, we show that $\Phi(A)$ is $\Phi$-continuous, that is, $\sup_{\Phi(A)}:\Phi(\Phi(A))\lra\Phi(A)$ has a left adjoint. To this end, write $\y_A$ for the Yoneda embedding $A\lra\Phi(A)$. For each $\Lambda\in\Phi(\Phi(A))$, since $\Lambda:\Phi(A) \lra(Q,d_R)$ preserves $\CQ$-order, it follows that  for all $x\in A$ and $\phi\in\Phi(A)$, \[\phi(x)= \sub_A(\y_A(x),\phi)\leq \Lambda(\phi)\ra\Lambda(\y_A(x)),  \] hence \[{\sup}_{\Phi(A)}(\Lambda)(x)= \bv_{\phi\in\Phi(A)}\Lambda(\phi)\&\phi(x) =\Lambda\circ\y_A(x) .\] This means that $\sup_{\Phi(A)}$ is the map obtained by restricting the domain and codomain of \[\y_A^\la:\CP\Phi(A)\lra\CP A\] to $\Phi(\Phi(A))$ and $\Phi(A)$, respectively. Therefore, $\sup_{\Phi(A)}$ has a left adjoint, given by restricting the domain and codomain of   $\y_A^\ra:\CP A\lra\CP\Phi(A)$  to  $\Phi(A)$ and $\Phi(\Phi(A))$, respectively.

(2) and (3) are a special case of  \cite[Theorem 4.7]{LZ07}, which is again  a special case of a general result in category theory \cite{AK,Kelly,KS05}. \end{proof}

The above theorem shows that if $\Phi$ is a saturated class of weights, then for each $\CQ$-ordered set $A$, $\Phi(A)$ is the free $\Phi$-continuous $\CQ$-ordered set generated by $A$.

This section concerns the saturatedness of the classes of forward Cauchy ideals, irreducible ideals, and flat ideals.

A quantale $\CQ=(Q,\&)$ is  completely distributive (continuous, resp.) if the complete lattice $Q$ is a  completely distributive lattice (a continuous lattice, resp.). So, each completely distributive quantale is a continuous quantale and each continuous quantale is a meet continuous quantale.   For continuity   and completely distributivity of complete lattices, the reader is referred  to the monograph \cite{Gierz2003}.

The following proposition was first proved in \cite{FSW} when $\CQ$ is a \emph{completely distributive value quantale}, the version presented below was proved in \cite{LZ06} making use of Lemma \ref{characterization of FC fuzzy lower set}.
\begin{prop} If $\CQ$ is a continuous quantale, then the class of forward Cauchy ideals is saturated. \end{prop}

\begin{lem} \label{characterization of FC fuzzy lower set} {\rm(\cite{LZ06})} Let $\CQ$ be a continuous quantale and $\phi$ be an inhabited fuzzy lower set of a $\CQ$-ordered set $A$.  The following are equivalent:
\begin{enumerate}[\rm(1)] \item   $\phi$ is a forward Cauchy ideal.
\item  If $r\ll\phi(x)$ and $s\ll\phi(y)$, then for every
$t\ll 1$, there is some $z\in A$ such that $t\ll\phi(z)$,
 $r\ll A(x,z)$  and $s\ll A(y,z)$.
 \end{enumerate}\end{lem}

The saturatedness of the classes of flat ideals and irreducible ideals   is a special case of a general result in enriched category theory, namely,  \cite[Proposition 5.4]{KS05}. However, in order to make this paper self-contained, a direct verification in this special case is included here.

\begin{prop}The class   of irreducible ideals is saturated. \end{prop}
\begin{proof}It suffices to show that for each $\CQ$-ordered set $A$ and each irreducible ideal $\Lambda:\mathcal{I} A\lra\CQ$ of $(\mathcal{I} A,\sub_A)$, the map $\sup\Lambda:A\lra\CQ $, given by $$\sup\Lambda(x) = \bv_{\phi\in\mathcal{I} A} \Lambda(\phi)\&\phi(x) ,$$ is an  irreducible ideal of $A$.

\textbf{Step 1}. $\bv_{x\in A}\sup\Lambda(x)=1$. This is easy since \[\bv_{x\in A}\sup\Lambda(x) = \bv_{x\in A}\bv_{\phi\in\mathcal{I} A} \Lambda(\phi)\&\phi(x)  = \bv_{\phi\in\mathcal{I} A}\bv_{x\in A} \Lambda(\phi)\&\phi(x) =\bv_{\phi\in\mathcal{I} A} \Lambda(\phi) =1. \]

\textbf{Step 2}. For any  fuzzy lower sets $\phi_1, \phi_2$ of $A$, \[\sub_A(\sup\Lambda,  \phi_1\vee \phi_2)= \sub_A(\sup\Lambda,  \phi_1)\vee \sub_A(\sup\Lambda,\phi_2).\]

To see this, for a  fuzzy lower set $\phi$ of $A$, consider the fuzzy lower set of $(\mathcal{I} A,\sub_A)$: \[\sub_A(-,\phi): \mathcal{I} A\lra\CQ .\]   Then
  \begin{align*}\sub_{\mathcal{I} A}(\Lambda,\sub_A(-,\phi))&= \bw_{\psi\in\mathcal{I} A} (\Lambda(\psi)\ra\sub_A(\psi,\phi))\\ &= \bw_{\psi\in\mathcal{I} A}\Big(\Lambda(\psi)\ra\bw_{x\in A}(\psi(x)\ra \phi(x))\Big)\\ &=\bw_{x\in A}\Big(  \Big(\bv_{\psi\in\mathcal{I} A}\Lambda(\psi)\& \psi(x)\Big)\ra \phi(x)\Big) \\   & =\sub_A(\sup\Lambda,\phi). \end{align*} Therefore, \begin{align*}
 \sub_A(\sup\Lambda,  \phi_1\vee \phi_2)
 &=\sub_{\mathcal{I} A}(\Lambda, \sub_A(-, \phi_1\vee \phi_2))\\
 &=\sub_{\mathcal{I} A}(\Lambda, \sub_A(-, \phi_1)\vee \sub_A(-,\phi_2)) \\
 &= \sub_{\mathcal{I} A}(\Lambda, \sub_A(-, \phi_1))\vee\sub_{\mathcal{I} A}(\Lambda, \sub_A(-, \phi_2))\\
 &=\sub_A(\sup\Lambda,  \phi_1)\vee \sub_A(\sup\Lambda,\phi_2),
\end{align*} where the second equality holds since each element in $\mathcal{I} A$ is irreducible; the reason for the third equality is that $\Lambda$ is irreducible.
 \end{proof}

\begin{prop}The class   of flat ideals is saturated. \end{prop}
\begin{proof} We only need to  show that for each $\CQ$-ordered set $A$ and each flat ideal $\Lambda:\CF A\lra\CQ$  of $(\CF A,\sub_A)$, the map  $\sup\Lambda:A\lra\CQ $, given by $$\sup\Lambda(x) = \bv_{\phi\in\CF A} \Lambda(\phi)\&\phi(x) ,$$ is a flat ideal of $A$.

\textbf{Step 1}. $\bv_{x\in A}\sup\Lambda(x)=1$. This is easy since \[\bv_{x\in A}\sup\Lambda(x) = \bv_{x\in A}\bv_{\phi\in\CF A} \Lambda(\phi)\&\phi(x)  = \bv_{\phi\in\CF A}\bv_{x\in A} \Lambda(\phi)\&\phi(x) =\bv_{\phi\in\CF A} \Lambda(\phi) =1. \]

\textbf{Step 2}. For all fuzzy upper sets $\psi_1, \psi_2$ of $A$, \[\sup\Lambda \otimes(\psi_1\wedge \psi_2)= (\sup\Lambda\otimes\psi_1)\wedge (\sup\Lambda\otimes\psi_2).\]

To see this, for each fuzzy upper set $\psi$ on $A$, consider the fuzzy upper set  of $(\CF A,\sub_A)$ (see Equation (\ref{compositon as distributor})):  \[-\otimes\psi: \CF A\lra\CQ. \]
Then \begin{align*}\Lambda\otimes(-\otimes \psi)
&= \bv_{\phi\in\CF A}\Big(\Lambda(\phi)\&\bv_{x\in A}(\phi(x)\& \psi(x))\Big)\\ &=\bv_{x\in A}  \bv_{\phi\in\CF A}(\Lambda(\phi)\& \phi(x))\& \psi(x) \\ &=\bv_{x\in A} \sup\Lambda(x)\& \psi(x)  \\ & =\sup\Lambda\otimes\psi. \end{align*} Therefore, \begin{align*}
 \sup\Lambda \otimes(\psi_1\wedge \psi_2)&=\Lambda\otimes( -\otimes(\psi_1 \wedge  \psi_2))\\
 &=\Lambda\otimes((-\otimes \psi_1)\wedge (-\otimes\psi_2))\\
 &= (\Lambda\otimes(-\otimes \psi_1))\wedge (\Lambda\otimes(-\otimes\psi_2))\\
 &=(\sup\Lambda\otimes\psi_1)\wedge( \sup\Lambda\otimes\psi_2).
\end{align*}

The proof is completed.
\end{proof}

\section{Scott $\CQ$-topology and Scott $\CQ$-cotopology}

The connection between partially ordered sets and topological spaces is the essence of domain theory. The fuzzy version of Alexandroff topology has been investigated in \cite{LZ06}. This section concerns the extension of Scott topology to the fuzzy setting.

We recall some basic definitions first.
A $\CQ$-topology  on a set $X$ is a subset $\tau$ of   $Q^X$  subject to the following conditions: \begin{enumerate} \item[(O1)] $p_X\in\tau$ for all $p\in Q$; \item[(O2)] $\lam\wedge \mu\in\tau$ for all $\lam,\mu\in\tau$; \item[(O3)] $\bv_{j\in J}\lam_j\in\tau$ for each subset $\{\lam_j\}_{j\in J}$ of $\tau$.\end{enumerate}

For a $\CQ$-topological space $(X,\tau)$, elements in $\tau$ are said to be open. A  $\CQ$-topology $\tau$ is \emph{stratified}   \cite{HS95,HS99} if \begin{enumerate} \item[(O4)] $p\&\lam \in\tau$ for all $p\in Q$ and $\lam\in \tau$.\end{enumerate}
A  $\CQ$-topology $\tau$ is \emph{co-stratified}   \cite{CLZ11} if \begin{enumerate} \item[(O5)] $p\ra\lam \in\tau$ for all $p\in Q$ and $\lam\in \tau$.\end{enumerate}
A  $\CQ$-topology   is \emph{strong} \cite{CLZ11,Zhang07} if it is both stratified and co-stratified.

A  $\CQ$-cotopology  on a set $X$ is a subset
$\tau$ of $Q^X$  subject to the following conditions:
\begin{enumerate}
\item[(C1)] $p_X\in\tau$ for all $p\in Q$;
\item[(C2)] $\lam\vee \mu\in\tau$ for all $\lam, \mu\in\tau$;
\item[(C3)]
$\bw_{j\in J} \lam_j\in\tau$ for each subset $\{\lam_j\}_{j\in J}$ of
$\tau$.\end{enumerate}

For a $\CQ$-cotopological space $(X,\tau)$, elements in $\tau$ are said to be closed. A  $\CQ$-cotopology $\tau$ is \emph{stratified} if
\begin{enumerate}
\item[(C4)] $p\ra \lam\in\tau$ for all $p\in Q$ and $\lam\in\tau$.
\end{enumerate}
A $\CQ$-cotopology $\tau$ is \emph{co-stratified} if \begin{enumerate} \item[(C5)] $p\&\lam\in\tau$ for all $p\in Q$ and $\lam\in\tau$. \end{enumerate}
A $\CQ$-cotopology $\tau$ is \emph{strong} \cite{CLZ11,Zhang07} if it is both stratified and co-stratified.

Let $\CQ$ be a quantale that satisfies the law of double negation. If $\tau$ is a stratified  (co-stratified, resp.) $\CQ$-cotopology on a set $X$, then  \[\neg(\tau)=\{\neg \lam \mid \lam\in\tau\}  \] is a stratified (co-stratified, resp.) $\CQ$-topology on $X$, where $\neg \lam (x)=\neg(\lam(x))$ for all $x\in X$. Conversely, if $\tau$ is a stratified  (co-stratified, resp.) $\CQ$-topology on  $X$, then   \[\neg(\tau)=\{\neg \lam \mid \lam\in\tau\}  \] is a stratified (co-stratified, resp.)  $\CQ$-cotopology on $X$.

In   general, there does not exist a natural way to switch between closed sets and open sets, so, we need to consider the open-set  version and the closed-set version when generalizing Scott topology to $\CQ$-ordered sets. As demonstrated below, flat ideals  and irreducible ideals are related to the open-set  and the closed-set version, respectively.

\begin{defn}
Let $\Phi$ be a class of weights and $A$ be a $\CQ$-ordered set. \begin{enumerate}[\rm(1)]
\item  A  fuzzy set  $\psi: A\lra Q$  is  $\Phi$-open  if it is a fuzzy upper set and  for all $\phi\in\Phi(A)$, \[\psi(\sup\phi)\leq \phi\otimes\psi
    \] whenever $\sup\phi$ exists. \item A fuzzy  set  $\lam:A\lra Q$   is  $\Phi$-closed if it is a fuzzy lower set and for  all $\phi\in\Phi(A)$, \[{\rm sub}_A(\phi,\lam)\leq \lam(\sup\phi) \]  whenever $\sup\phi$ exists. \end{enumerate} \end{defn}

Let $\psi$ be a fuzzy upper set  of $A$. Since $\phi\otimes\psi=\sup\psi^\ra(\phi)$ by Example \ref{intersection as sup}, then $\phi\otimes\psi\leq \psi(\sup\phi)$, hence a fuzzy upper set $\psi$  is $\Phi$-open if and only if \begin{equation*}
\psi(\sup\phi)= \phi\otimes\psi\end{equation*}for all $\phi\in\Phi(A)$,  if and only if $\psi:A\lra(Q,d_L)$  is $\Phi$-cocontinuous in the sense that $\psi(\sup\phi)=\sup\psi^\ra(\phi)$ for all $\phi\in\Phi(A)$. Similarly,   a fuzzy lower set $\lam$   is $\Phi$-closed if and only if \begin{equation}\label{Scott closed =}{\rm sub}_A(\phi,\lam)= \lam(\sup\phi)\end{equation} for all $\phi\in\Phi(A)$,  if and only if $\lam: A\lra (Q,d_R)$  is $\Phi$-cocontinuous.

\begin{prop} \label{open sets} Let $\Phi$ be a class of weights. \begin{enumerate}[(1)] \item Each constant fuzzy set is $\Phi$-open. \item If $\psi$ is a $\Phi$-open, then so is $p\&\psi$ for all $p\in Q$. \item The join  of a set of $\Phi$-open fuzzy sets  is   $\Phi$-open. \item If   $\Phi$ is a subclass of flat ideals, then the meet of two  $\Phi$-open fuzzy sets   is   $\Phi$-open.
\end{enumerate} \end{prop}
Thus, if $\Phi$ is a subclass of flat ideals, then for each $\CQ$-ordered set $A$, the $\Phi$-open fuzzy sets of $A$ form a  stratified $\CQ$-topology  on $A$, called the $\Phi$-Scott $\CQ$-topology and  denoted by $\sigma_\Phi(A)$.

\begin{prop} \label{closed sets} Let $\Phi$ be a class of weights. \begin{enumerate}[(1)] \item Each constant fuzzy set is $\Phi$-closed. \item If $\psi$ is a $\Phi$-closed fuzzy set of $A$, then so is $p\ra\psi$ for all $p\in Q$. \item The meet  of a set of $\Phi$-closed fuzzy sets  is $\Phi$-closed. \item If   $\Phi$ is a subclass of irreducible ideals, then the join  of two $\Phi$-closed fuzzy sets is $\Phi$-closed.
\end{enumerate} \end{prop}

Thus, if  $\Phi$ is a subclass of irreducible ideals, then  for each $\CQ$-ordered set $A$, the $\Phi$-closed fuzzy sets form a stratified $\CQ$-cotopology on $A$, called the $\Phi$-Scott $\CQ$-cotopology on $A$ and  denoted by $\sigma^{\rm co}_\Phi(A)$.


\begin{con} For the class $\CF$ of all flat ideals, we say fuzzy Scott open sets (Scott $\CQ$-topology, resp.) instead  of $\CF$-open fuzzy sets ($\CF$-Scott $\CQ$-topology, resp.), and write $\sigma(A)$ instead of $\sigma_\CF(A)$. Dually, For the class $\CI$ of all irreducible ideals, we   say fuzzy Scott closed sets (Scott $\CQ$-cotopology, resp.) instead  of $\CI$-closed fuzzy sets ($\CI$-Scott $\CQ$-cotopology, resp.), and write $\sigma^{\rm co}(A)$ instead of $\sigma^{\rm co}_\CI(A)$. \end{con}

\begin{rem}  (1)  Let $\{x_i\}$ be a forward Cauchy net in  a $\CQ$-ordered set $A$  and \[\varphi=\bv_i\bw_{j\geq i}A(-,x_j).\] For each fuzzy upper set $\psi$ of $A$, by the argument of Theorem \ref{FC is flat}, we have \[\varphi\otimes\psi=\bv_i\bw_{j\geq i}\psi(x_j).\]  Thus, a fuzzy upper set $\psi$ of $A$ is $\mathcal{W}$-open if and only if  \[\bv_i\bw_{j\geq i}\psi(x_j)\geq\psi(x)\] for every forward Cauchy net $\{x_i\}$ with a Yoneda limit $x$. This shows that $\mathcal{W}$-open fuzzy sets are  the Scott open fuzzy sets  in the sense of Wagner  \cite[Definition 4.1]{Wagner97}. In particular, if $\CQ$ is meet continuous,  then   $\mathcal{W}$-open fuzzy sets  of $A$ form a stratified $\CQ$-topology on $A$.

(2) Lemma 4.6 in \cite{Wagner97} claims that if $\phi,\psi$ are $\mathcal{W}$-open fuzzy sets of a $\CQ$-ordered set $A$, then so is the fuzzy set $\phi\&\psi:A\lra Q$ given by $(\phi\&\psi)(x)= \phi(x)\&\psi(x)$. This is not true in general. Let $\CQ$ be the unit interval $[0,1]$ equipped with the product t-norm $\&$. For each class of weights $\Phi$, the identity map $\id$ is clearly $\Phi$-open in the $\CQ$-ordered set $([0,1],d_L)$, but, $\id\&\id$ is not a fuzzy upper set of $([0,1],d_L)$.

(3) If $\CQ=(Q,\&)$ is a frame, i.e., $\&=\wedge$, then, as noted in Remark \ref{history}, the fuzzy ideals considered in \cite{Yao16} are exactly the flat ideals, hence a fuzzy set $\psi$ of a $\CQ$-ordered set $A$ is fuzzy Scott open if and only if it is so in the sense of Yao  \cite[Definition 2.10]{Yao16}. \end{rem}

\begin{rem}Let $\{x_i\}$ be a forward Cauchy net in  a $\CQ$-ordered set  $A$ and \[\phi=\bv_i\bw_{j\geq i}A(-,x_j).\] By Proposition \ref{3.9} (or, \textbf{Step 2} in the argument of Theorem \ref{FC is irreducible}), \[\sub_A(\phi,\psi)= \bv_i\bw_{j\geq i}\psi(x_j)\] for each fuzzy lower set $\psi$ of $A$. By Proposition \ref{yoneda limit as suprema}, Yoneda limits of $\{x_i\}$ are exactly the suprema of the fuzzy lower set \[\phi=\bv_i\bw_{j\geq i}A(-,x_j).\] So, a fuzzy lower set $\psi$ of $A$ is $\mathcal{W}$-closed if and only if \[\bv_i\bw_{j\geq i}\psi(x_j)\leq\psi(x) \] for every forward Cauchy net $\{x_i\}$  with a Yoneda limit $x$. This shows that $\mathcal{W}$-closed fuzzy sets are exactly the Scott closed fuzzy sets in the sense of Wagner (\cite{Wagner97}, Definition 4.4). \end{rem}

\begin{prop}Let $\Phi$ be a subclass of flat ideals. Then for each $\Phi$-cocontinuous map $f: A\lra B$ between $\CQ$-ordered sets, $f:(A,\sigma_\Phi(A))\lra (B,\sigma_\Phi(B))$ is continuous. \end{prop}

Therefore, for a subclass $\Phi$ of flat ideals,  assigning each $\CQ$-ordered set $A$ to  the $\CQ$-topological space $(A,\sigma_\Phi(A))$ defines a functor $\Sigma_\Phi$ from the category of $\CQ$-ordered sets and $\Phi$-cocontinuous maps to that of stratified $\CQ$-topological spaces. It is known in domain theory that the functor sending each  ordered set to its Scott topology is a full functor, but, it is not clear whether $\Sigma_\Phi$ is a full functor.

The situation with Scott $\CQ$-cotopology looks more promising. The following conclusion implies that for every subclass $\Phi$ of irreducible ideals,  assigning each $\CQ$-ordered set $A$ to  the $\CQ$-cotopological space $(A,\sigma^{\rm co}_\Phi(A))$ gives a  full functor $\Sigma^{\rm co}_\Phi$ from the category of $\CQ$-ordered sets and $\Phi$-cocontinuous maps to that of stratified $\CQ$-cotopological spaces.

\begin{prop}\label{full functor} {\rm(\cite[Proposition 4.15]{Wagner97} for the class of forward Cauchy ideals)} Let $\Phi$ be a  class of   weights. For each map    $f: A\lra B $  between $\CQ$-ordered sets, the following are equivalent: \begin{enumerate}[(1)] \item $ f: A\lra B $ is $\Phi$-cocontinuous.
   \item For each $\Phi$-closed fuzzy set $\phi$ of $B$, $\phi\circ f$ is  a $\Phi$-closed fuzzy set of  $A$. \end{enumerate}\end{prop}
\begin{proof}
$(1)\Rightarrow(2)$
This is easy since the composite of  $\Phi$-cocontinuous maps is $\Phi$-cocontinuous.

$(2)\Rightarrow(1)$ First, we show that $f$ preserves $\CQ$-order. For all $a_1,a_2\in A$, since  $\psi=B(-,f(a_2))$ is a $\Phi$-closed fuzzy set of $B$, then $\psi\circ f=B(f(-),f(a_2))$ is $\Phi$-closed, hence
\[A(a_1,a_2)   =\psi\circ f(a_2)\&A(a_1,a_2) \leq \psi\circ f(a_1)= B(f(a_1),f(a_2)), \] showing that $f$ preserves $\CQ$-order.

Second, we show that for each   $\phi\in\Phi(A)$, if $\sup\phi $ exists, then for all $b\in B$,
\[\sub_B(f^\ra(\phi),B(-,b))=B(f({\sup}\phi),b),\] hence $f(\sup\phi)$ is a supremum of $f^\ra(\phi)$.
Since $B(-,b)$ is a $\Phi$-closed fuzzy set of $B$,   $B(-,b)\circ f$ is a $\Phi$-closed fuzzy set of $A$, hence, by Eq. (\ref{Scott closed =}), \[\sub_B(f^\ra(\phi),B(-,b))=\sub_A(\phi,B(-,b)\circ f) = B(f({\sup}\phi),b).\]
This completes the proof.
\end{proof}

\begin{exmp}  \label{5.8}
Let $\CQ=([0,1],\&)$ with $\&$ being a left continuous t-norm. Then $\phi$   is a fuzzy Scott closed set in  $([0,1],d_R)$ if and only if  $\phi:([0,1],d_L)\lra([0,1],d_L)$ is right continuous and  $\CQ$-order preserving.

By  Corollary \ref{irreducible ideals in [0,1]}, a fuzzy lower set $\psi$  of $([0,1],d_R)$  is an irreducible ideal if and only if either  $\psi(x)=a \rightarrow x$ for some $a\in [0,1]$ or $\psi(x)=\bv_{b>a}(b\ra x)$ for some $ a<1$. Since the supremum of $ \bv_{b>a}(b\ra x)$ in $([0,1],d_R)$  is (see Example \ref{inclusion as sup}) \[\bw_{x\in [0,1]}\Big(\bv_{b>a}(b\ra x)\ra x\Big)=\bw_{b>a}\bw_{x\in [0,1]}((b\ra x)\ra x)= a,\] it follows that a fuzzy lower set $\phi$  of $([0,1],d_R)$ is fuzzy Scott closed  if and only if for all $a<1$, \[\sub_{[0,1]}\Big(\bv_{b>a}(b\ra x),\phi\Big)= \bw_{b>a}\phi(b)\leq\phi(a).\] The conclusion thus follows. \end{exmp}

If $\CQ$ is the unit interval $[0,1]$ equipped with a continuous t-norm  $\&$, we have a bit more: the fuzzy Scott $\CQ$-cotopology on each $\CQ$-ordered set is a strong $\CQ$-cotopology.

\begin{prop}\label{strong} Let $\CQ=([0,1],\&)$ with $\&$ being a left continuous t-norm. The following are equivalent: \begin{enumerate}[\rm (1)] \item $\&$ is a   continuous t-norm. \item The Scott $\CQ$-cotopology on each $\CQ$-ordered set is a strong $\CQ$-cotopology. \end{enumerate} \end{prop}

\begin{proof}$(1)\Rightarrow(2)$ We only need to check that if $\phi$ is a fuzzy Scott closed fuzzy set of a $\CQ$-ordered set $A$, then so is $a\&\phi$ for all $a\in[0,1]$. Since $\phi$ is a fuzzy Scott closed set of $A$ and $a\&\id$ is a fuzzy Scott closed set of $([0,1],d_R)$, both $\phi:A\lra([0,1],d_R)$ and $a\&\id:([0,1],d_R)\lra([0,1],d_R)$ preserve suprema of irreducible ideals, then  $a\&\phi=(a\&\id)\circ\phi:A\lra([0,1],d_R)$ preserve suprema of irreducible ideals, hence $a\&\phi$ is fuzzy Scott closed.

$(2)\Rightarrow(1)$ If $\&$ is not continuous, by  \cite[Proposition 1.19]{KMP00} there is some $a\in[0,1]$ such that $a\&{\rm id}:[0,1]\lra[0,1]$ is not right continuous, hence not a fuzzy Scott closed set in $([0,1],d_R)$. Since the identity map  on $[0,1]$ is   fuzzy Scott closed in $([0,1],d_R)$ by Example \ref{5.8},   the Scott $\CQ$-cotopology on $([0,1],d_R)$ cannot be a strong one.
\end{proof}

\begin{exmp}This example shows that if $\CQ=([0,1],\&)$ with $\&$ being a continuous t-norm, then the Scott $\CQ$-cotopology on $([0,1],d_R)$ is the strong $\CQ$-cotopology on $[0,1]$ generated by the identity map.

Let $\tau$ denote the strong $\CQ$-cotopology  on $[0,1]$ generated by the identity map. By Example \ref{5.8}, a fuzzy Scott closed set in  $([0,1],d_R)$ is exactly a right continuous and  $\CQ$-order preserving map $\phi:([0,1],d_L)\lra([0,1],d_L)$. So, the conclusion has already been proved in \cite{Zhang18} in the case that $\&$ is  the t-norm $\min$, the product t-norm,  and the {\L}ukasiewicz t-norm. Here we prove it in the general case by help of the ordinal sum decomposition of continuous t-norms. Since the Scott $\CQ$-cotopology  on $([0,1],d_R)$ is  strong  and contains the identity map as a closed set, it suffices to show that if $\phi$ is a fuzzy Scott closed set in  $([0,1],d_R)$ then $\phi\in\tau$.  We do this in two steps.

\textbf{Step 1}. If $\phi$ is a fuzzy Scott closed set in  $([0,1],d_R)$ and  $\phi\geq \id$, then $\phi\in\tau$.

Since $\&$ is a continuous t-norm,  there is a set of disjoint open intervals $\{(a_i,b_i)\}$ such that \begin{itemize}\setlength{\itemsep}{-2pt} \item for each $i$, both $a_i$ and $b_i$ are idempotent and the restriction of $\&$ on $[a_i,b_i]$ is either isomorphic to the \L ukasiewicz t-norm or to the product t-norm; \item $x\&y=\min\{x,y\}$ if $(x,y)\notin\bigcup_i[a_i,b_i]^2$.  \end{itemize}

For each $x\in[0,1]$, define $g_x: [0,1]\lra[0,1]$ by \[ g_x(y)=\begin{cases}
\phi(x)\vee((\phi(x)\ra x)\ra y),&\text{$ (x,\phi(x))\in(a_i,b_i)^2 $  for some $i$ and $\phi(x)>x$,} \\
\phi(x)\vee(b_i\ra y),&\text{$ (x,\phi(x))\in(a_i,b_i)^2 $  for some $i$ and $ \phi(x)=x $,}\\
\phi(x)\vee(x\ra y),&\text{$ (x,\phi(x))\not\in(a_i,b_i)^2 $  for any $ i $.}
\end{cases} \]

Each $g_x$ is clearly a member of $\tau$, so,  in order to see that $\phi\in\tau$, it suffices to  show that  for all $y\in[0,1]$, \[\phi(y)=\bw_{x\in[0,1]}g_x(y).\] Before proving this equality, we list here some  facts about the maps $g_x$, the verifications are left to the reader. \begin{enumerate}[(M1)]
\item $\phi(y)\leq g_x(y)$ whenever $y\leq x$. \item If $ (x,\phi(x))\in(a_i,b_i)^2 $ for some $i$ and $\phi(x)>x$, then  $g_x(x)= (\phi(x)\ra x)\ra x=\phi(x)$  and $g_x(y)= (\phi(x)\ra x)\ra y\geq\phi(y)$ for all $y>x$.
\item  If $(x,\phi(x))\in(a_i,b_i)^2$ for some $i$ and $\phi(x)=x$,  then $\phi(y)=y=g_x(y)$ whenever $x\leq y<b_i$ and $g_x(y)=1\geq\phi(y)$ for all $y\geq b_i$.
\item  If $(x,\phi(x))\not\in(a_i,b_i)^2$  for any $i$, then for all $y\geq x$, $g_x(y)=\phi(x)\vee (x\ra y)=1 \geq \phi(y)$. \end{enumerate}

It follows immediately from these facts that  for all $y\in[0,1]$, $\phi(y)\leq\bw_{x\in[0,1]}g_x(y)$. For the converse inequality, we distinguish three cases.

\textbf{Case 1}.  $(y,\phi(y))\in(a_i,b_i)^2 $  for some $i$ and $ \phi(y)>y$. Then by fact (M2), $g_{y}(y)=\phi(y)$, hence $\phi(y)\geq\bw_{x\in[0,1]}g_x(y)$.

\textbf{Case 2}.   $(y,\phi(y))\in(a_i,b_i)^2$  for some $i$ and $\phi(y)=y$. Then by fact (M3), $g_{y}(y)=\phi(y)$, hence $\phi(y)\geq\bw_{x\in[0,1]}g_x(y)$.

\textbf{Case 3}.    $ (y,\phi(y))\not\in(a_i,b_i)^2 $  for any $i$. In this case, if we can show that $g_x(y)=\phi(x)$ for all $x>y$, then we will obtain that $\phi(y)=\bw_{x>y}\phi(x)\geq\bw_{x\in[0,1]}g_x(y)$ by right continuity of $\phi$. The proof is divided into four subcases.

Subcase 1.  $y \in (a_i,b_i) $ for some $i$ and $\phi(y)\geq b_i$. If $x\leq b_i $, then $x\ra y \leq  b_i$ and $\phi(x)\geq\phi(y)\geq b_i $, hence  $g_x(y)=\phi(x)\vee (x\ra y)=\phi(x)$. For $x>b_i$,  \begin{itemize}\setlength{\itemsep}{-2pt}
		\item if $ (x,\phi(x))\in(a_j,b_j)^2 $  for some $j$ and $ \phi(x)>x $, then \[g_x(y)=\phi(x)\vee((\phi(x)\ra x)\ra y)=\phi(x)\vee y =\phi(x);\]
		\item if $ (x,\phi(x))\in(a_j,b_j)^2 $  for some $j$ and $ \phi(x)=x $, then \[g_x(y)=\phi(x)\vee(b_j\ra y)=\phi(x)\vee y =\phi(x);\]
		\item if $(x,\phi(x))\not\in(a_j,b_j)^2$  for any $j$, then \[g_x(y)=\phi(x)\vee(x\ra y)=\phi(x)\vee y =\phi(x).\]
		\end{itemize}

Subcase 2.  $y \not\in [a_i,b_i]$ for any $i$. In this case, since $t\ra y= y$ for all $t>y$, it follows that $g_x(y)=\phi(x)\vee y=\phi(x)$.

Subcase 3.  $y=a_i$. For $x<b_i$, \begin{itemize}\setlength{\itemsep}{-2pt}
			\item if $ x<\phi(x)<b_i $, then $(x,\phi(x))\in(a_i,b_i)^2$, hence $g_x(y)=\phi(x)\vee((\phi(x)\ra x)\ra a_i)=\phi(x)$;
			\item if $ \phi(x)\geq b_i $, then $g_x(y)=\phi(x)\vee(x\ra y)=\phi(x)$ since $x\ra y=x\ra a_i \leq b_i$;
			\item if $ \phi(x)=x $, then $g_x(y)=\phi(x)\vee(b_i\ra y)=\phi(x)\vee(b_i\ra a_i)=\phi(x)$.
		\end{itemize}
For $x>b_i$,
		\begin{itemize}\setlength{\itemsep}{-2pt}
				\item if $(x,\phi(x))\in(a_j,b_j)^2$  for some $j$ and $\phi(x)>x$, then $a_i<a_j$, hence \[g_x(y)=\phi(x)\vee((\phi(x)\ra x)\ra a_i)=\phi(x)\vee a_i =\phi(x);\]
				\item if $(x,\phi(x))\in(a_j,b_j)^2$  for some $j$ and $\phi(x)=x$, then $g_x(y)=\phi(x)\vee(b_j\ra a_i)=\phi(x)$;
				\item if $(x,\phi(x))\not\in(a_j,b_j)^2$  for any $j$, then $g_x(y)=\phi(x)\vee(x\ra a_i)=\phi(x)$.
			\end{itemize}

Subcase 4.  $y=b_i$.  If $a_j=b_i$ for some $j$, then the conclusion holds by  Subcase 3.  Otherwise, the argument for Subcase 2 can be applied to show that $g_x(y)=\phi(x)$.

\textbf{Step 2}. If $\phi$ is a fuzzy Scott closed set in  $([0,1],d_R)$, then $\phi\in\tau$.

Since $\phi(1)\ra \phi$ is fuzzy Scott closed and $\id\leq \phi(1)\ra \phi$, it follows that $\phi(1)\ra \phi\in\tau$ by Step 1.  Since $\tau$ is strong and $\&$ is continuous, then $\phi=\phi(1)\&(\phi(1)\ra \phi) \in\tau$.  \end{exmp}

\end{document}